\documentclass[oneside]{amsart}
\usepackage[utf8]{inputenc}

\usepackage{amssymb}
\usepackage[abbrev]{amsrefs}

\usepackage{hyperref}
\usepackage{enumitem} 

 \usepackage{todonotes}

\usepackage{amsmath,euscript}
\usepackage{amsthm}
\usepackage{graphicx}
\usepackage{mathrsfs}
\usepackage{epsf}
\usepackage{xcolor}
\usepackage{psfrag}
\usepackage{amsfonts,amssymb,amsthm,amsmath} 

\usepackage{epsfig}
\usepackage{graphicx}
\usepackage{color}

\usepackage[toc,page]{appendix}

\usepackage{xspace}

\usepackage[margin=1.77in]{geometry}

\usepackage{changes}
\definechangesauthor[name={JZ}, color=blue]{jf}
\definechangesauthor[name={NE}, color=red]{ml}
\definechangesauthor[name={MS}, color=green]{ms}

\definecolor{green}{rgb}{0.0, 0.5, 0.5}
\definecolor{yellow}{rgb}{0.5, 0.5, 0}
\definecolor{lgray}{gray}{0.9}
\definecolor{llgray}{gray}{0.95}
\definecolor{lllgray}{gray}{0.975}

\usepackage{soul}
\setstcolor{red}

\theoremstyle{plain}%
\newtheorem{theorem}{Theorem}
\newtheorem{proposition}[theorem]{Proposition}%
\newtheorem{lemma}[theorem]{Lemma}
\newtheorem{corollary}[theorem]{Corollary}
\newtheorem{conjecture}[theorem]{Conjecture}

\theoremstyle{remark}%
\newtheorem{remark}{Remark}%

\theoremstyle{definition}%
\newtheorem{definition}[theorem]{Definition}%

\numberwithin{theorem}{section}
\numberwithin{equation}{section}

\newcommand{\inn}[2]{\left\langle#1,\,#2\right\rangle}
\DeclareMathOperator{\ran}{ran}

\DeclareMathOperator{\Null}{Null}
\DeclareMathOperator{\supp}{supp}

\DeclareMathOperator{\diag}{diag}

\DeclareMathOperator{\sgn}{sign}

\newcommand{\grad}{\nabla}
\newcommand{\Lap}{\Delta}
\newcommand{\di}{\partial}

\newcommand{\si}{\sigma}
\newcommand{\eps}{\epsilon}
\newcommand{\ls}{\lesssim}
\newcommand{\g}{\gamma}
\newcommand{\al}{\alpha}
\newcommand{\Si}{\Sigma}

\renewcommand{\b}{\bar}
\newcommand{\Cb}{\mathbb{C}}

\newcommand{\Zb}{\mathbb{Z}}
\newcommand{\Rb}{\mathbb{R}}

\renewcommand{\H}{\mathcal{H}}
\newcommand{\Lc}{\mathcal{L}}

\newcommand{\Sc}{\mathcal{S}}

\newcommand{\Hb}{\mathbb{H}}

\newcommand{\om}{\omega}
\newcommand{\Om}{\Omega}
\newcommand{\Ga}{\Gamma}

\renewcommand{\l}{\lambda} 
\renewcommand{\b}{\bar} 

\newcommand{\abs}[1]{\left\lvert#1\right\rvert}
\newcommand{\norm}[1]{\left\lVert#1\right\rVert}

\newcommand{\br}[1]{\langle#1\rangle}
\newcommand{\Set}[1]{\left\{#1\right\}}
\newcommand{\fullfunction}[5]{\ensuremath{
		\begin{array}{ccrcl}
			{#1}    & \colon  & {#2} & \longrightarrow & {#3} \\
			\mbox{} & \mbox{} & {#4} & \longmapsto     & {#5}
\end{array}}}

\newcommand{\md}[6]{\ensuremath{
		\ifinner
		\tfrac{\partial{^{#2}}#1}{\partial{#3^{#4}}\partial{#5^{#6}}}
		\else
		\tfrac{\partial{^{#2}}#1}{\partial{#3^{#4}}\partial{#5^{#6}}}
		\fi
}}
\newcommand{\del}[1]{\left(#1\right)}
\newcommand{\thmref}[1]{Theorem~\ref{#1}}

\newcommand{\secref}[1]{Section~\ref{#1}}
\newcommand{\lemref}[1]{Lemma~\ref{#1}}
\newcommand{\propref}[1]{Proposition~\ref{#1}}
\newcommand{\remref}[1]{Remark~\ref{#1}}
\newcommand{\figref}[1]{Figure~\ref{#1}}
\newcommand{\corref}[1]{Corollary~\ref{#1}}

\newcommand{\ba}[1]{\begin{array}{#1}}
	\newcommand{\ea}{\end{array}}

\newcommand{\be}{\begin{equation}}
	\newcommand{\ee}{\end{equation}}
\newcommand{\bea}{\begin{eqnarray}}
	\newcommand{\eea}{\end{eqnarray}}
\newcommand{\beann}{\begin{eqnarray*}}
	\newcommand{\eeann}{\end{eqnarray*}}

\newcommand{\R}{\mathbb{R}}
\newcommand{\C}{\mathbb{C}}
\newcommand{\Z}{\mathbb{Z}}
\newcommand{\T}{\mathbb{T}}

\newcommand{\bH}{{\mathbb{H}}}

\newcommand{\cE}{\mathcal{E}}

\newcommand{\cH}{\mathcal{H}}

\newcommand{\cL}{\mathcal{L}}         

\newcommand{\cS}{\mathcal{S}}

\newcommand{\nc}{\newcommand}

\nc{\bet}{\beta}
\nc{\G}{\Gamma}

\nc{\lam}{\lambda}

\nc{\ta}{\tau}
\nc{\w}{\omega}
\nc{\io}{\iota}
\nc{\h}{\theta}
\nc{\s}{\sigma}
\nc{\Lam}{\Lambda}

\nc{\vphi}{\varphi}

\nc{\Omt}{\tilde{\Omega}}

\newcommand{\re}{\operatorname{Re}}
\newcommand{\im}{\operatorname{Im}}
\renewcommand{\Re}{\operatorname{Re}}
\renewcommand{\Im}{\operatorname{Im}}

\nc{\lra}{\leftrightarrow}
\nc{\ra}{\rightarrow}
\nc{\ff}{\Leftrightarrow}

\newcommand{\one}{{\bf{1}}}
\nc{\bfone}{{\bf 1}}

\nc{\p}{\partial}
\nc{\pt}{\partial_t}
\nc{\ptt}{\partial_t^2}
\nc{\dt}{\partial_t}

\newcommand{\n}{\nabla}
\nc{\dA}{\nabla_A}

\newcommand{\DETAILS}[1]{}

\usepackage{tikz}

\begin{document}
	\title[GLE on non-compact Riemann surfaces]{Ginzburg-Landau Equations on Non-compact Riemann Surfaces}

\author[N. M. Ercolani]{Nicholas M. Ercolani}
\address{Department of Mathematics, University of Arizona, Tucson, AZ 85721-0089, USA}
\email{ercolani@math.arizona.edu}

\author[I. M. Sigal]{Israel Michael Sigal}
\address{Department of Mathematics, University of Toronto, Toronto, ON M5S 2E4, Canada }
\email{im.sigal@utoronto.ca}

	\author[J. Zhang]{Jingxuan Zhang
}
\address{Department of Mathematical Sciences, University of Copenhagen, Copenhagen 2100, Denmark}
\address{Department of Mathematics, University of Toronto, Toronto, ON M5S 2E4, Canada}
\email{jingxuan.zhang@math.ku.dk}

\date{April 17, 2023}
\subjclass[2020]{35Q56 (primary); 35J66 (secondary)  }
\keywords{Ginzburg–Landau equations; Superconductivity; Bifurcation theory; Elliptic equations on Riemann surfaces}
\begin{abstract}

	We study the Ginzburg-Landau equations on line bundles over non-compact Riemann surfaces with constant negative curvature. We prove existence of solutions with energy strictly less than that of the constant curvature (magnetic field) one. These solutions are the non-commutative generalizations of the Abrikosov vortex lattice of superconductivity. Conjecturally, they are (local) minimizers of the Ginzburg-Landau energy. We obtain precise asymptotic expansions of these solutions and their energies in terms of the curvature of the underlying Riemann surface. Among other things, our result shows the spontaneous breaking of the gauge-translational symmetry of the Ginzburg-Landau equations.
\end{abstract}

	\maketitle

	\section{Introduction}\label{sec:1}
%
%
%
%
%
%

We consider the Ginzburg-Landau equations on a line bundle $E$
over a Riemann surface $\Si$ with a Hermitian metric $h$:
\begin{equation}
	\label{GL}
	\tag{GL}
	\begin{aligned}
		-\Lap_a\psi&=\kappa^2\del{1-\abs{\psi}^2}\psi,\\
		d^*d a&=\Im \del{\bar{\psi}\grad_a\psi}.
	\end{aligned}
\end{equation}
Here $\kappa>0$ is a fixed parameter.
$\psi$ and  $a$ are respectively a section of and a connection $1$-form on the line bundle $E$.
$\grad_a$ is the covariant derivative induced by the $1$-form $a$, and 
$-\Lap_a=\grad_a^*\grad_a$ is the covariant Laplacian,
{both acting on sections of $E$.}
$d$ denotes the  exterior derivative on $\Si$.
Note that the adjoint $\grad_a^*$ depends on the metric $h$.
See Appendix \ref{sec:A} for detailed definitions.

{In the standard Ginzburg-Landau equations, $\kappa$ is the dimensionless Ginzburg-Landau material parameter, $\psi$ the complex order parameter for the electronic condensate on $\Si$, and $a$ the vector potential, with the $2$-form $da$ giving the 
magnetic field,  $\abs{\psi}^2$ the local density of superconducting
electrons, and $J(\psi,a):=\Im (\bar{\psi}\grad_a\psi)$ the supercurrent density.}  

\eqref{GL} are the Euler-Lagrange equations for the
Ginzburg-Landau energy,
\begin{equation}
	\label{1.1}
	\cE(\psi,a,h)=\int_\Si \del{\frac{1}{2}\abs{\grad_a\psi}^2+\frac{1}{2}\abs{da}^2+\frac{\kappa^2}{4}\del{\abs{\psi}^2-1}^2}\,\om.
\end{equation}
The Hermitian metric $h$ enters \eqref{1.1} through the area $2$-form $\om$ induced by $h$. 

{Physically, \eqref{1.1} corresponds to the Ginzburg-Landau Helmholz free
energy. By the Chern-Weil correspondence (see \secref{sec:2.5} below), $\cE(\psi,a,h)$ can be parametrized by the average magnetic field in the sample.
It is related to the Ginzburg-Landau Gibbs free energy, depending
on applied magnetic field through the Legendre transform.
For more discussions, see \cite{MR3381531}. }

 {One can think of solutions to \eqref{GL} as non-commutative versions of the Abrikosov vortex lattices, with commutative lattice $\cL$ acting on $\Cb$ by translations replaced by a non-commutative one -- a Fuchsian group $\Ga$ acting on the Poincar\'e half-plane $\Hb$. See \secref{sec:2} for details. 
 	
 	One can also connect \eqref{GL} to the Ginzburg-Landau equations on a thin superconducting membrane,
 and we conjecture that the mathematical techniques developed in this 
 paper can be applied to this latter model. See \cite{MR1426137}
 for a review of the physics problem.}

 		\eqref{GL} is the first and arguably the simplest gauge theory. Indeed, \eqref{GL} is invariant under local $U(1)$-gauge transforms
 		\begin{equation}
 			\label{gaugetransf}
 			(\psi,a)\mapsto (g\psi,a+g^{-1}dg),
 		\end{equation}
 		where $g$ is a $U(1)$-valued isomorphism 
 		\footnote{The fact that \eqref{gaugetransf} is indeed a symmetry for \eqref{GL} follows from the relations
 			$\grad _a(g\psi)=g\grad_a\psi+(dg)\psi=g(\grad_a+g^{-1}dg)\psi=g\grad_{a+g^{-1}dg}\psi
 			$, and 
 			$\Im(\overline{g\psi}\grad_a(g\psi))=\Im(\overline{g\psi}g\grad_{a+g^{-1}dg}\psi)=\Im (\b\psi g\grad_{a+g^{-1}dg}\psi).$
 		}
		{and $a$ is a gauge field related to the standard connection on the principal $U(1)$-bundle.}

In this paper,
{we construct nontrivial solutions  to the Ginzburg-Landau equations \eqref{GL} defined on a unitary line bundle $E$ over} a non-compact Riemann surface $\Si$ of finite volume. {
Our existence theory holds on the arithmetic surfaces}
\begin{equation}\label{1.2}
	\Sigma\cong \Hb /\Ga(N),\quad N=1,2,\ldots,
\end{equation}
where $\Hb$ is the Poincar\'e half-plane,  equipped with a hyperbolic metric
with constant negative curvature,
and 
$\Ga(N)$ is the principal congruence subgroup of level
$N$, {acting on $\Hb$ by M\"obius transform.} 
\eqref{GL} on a non-compact Riemann surface
$\Si$  of the form \eqref{1.2} could serve as a toy model of a superconducting circuit with
several narrow open ends \cite{Pnueli}. 
The precise results are given in Thms.~\ref{thm1.2}--\ref{thm1.1}.

%
%
%

We also obtain asymptotics for the energy of the our solutions and we show that under condition \eqref{1.100}, the energy of the our solutions is lower than that of the constant curvature solution (see \eqref{1.3} below). The precise energy estimate is given in \corref{cor1.2}.

{Throughout the paper, we fix some non-compact Riemann surface
	$(\Si,h)$, together with a unitary line bundle $E\to\Si$,
	and then seek a solution pair $(\psi,a)$ consisting of 
	a section of and a connection on $E$.}
The parameter $\kappa>0$ in \eqref{GL} is a
	fixed  number. Unless otherwise stated, dependence of
	various quantities on $\kappa$ is {not displayed but} always
	understood.

\subsection{Main Results}
\label{sec:1.1}
To state our main result, 
	we introduce some definitions. To fix the ideas, we consider the family of {hyperbolic}
metrics on $\Si$ given by
\begin{equation}
	\label{hr}
	h_r = \frac{r}{(\Im z)^2} dz \otimes d\bar{z}\quad  (r>0).
\end{equation}
The corresponding area $2$-form is given by
\begin{equation}\label{omr}
	\om_r=\frac{ri}{2(\Im z)^2}dz\wedge d\bar z.
\end{equation}
Each $h_r$ turns $\Si$ into a surface with constant curvature
$-1/r$. 
Denote 
$\abs{\Si}$ the area of $\Si$ w.r.t.~the standard hyperbolic $2$-form $\om_1\equiv \tfrac{i}{2(\Im z)^2}dz\wedge d\bar z.$ Then the surface $(\Si,h_r)$ has area $\abs{\Si}_r=r\abs{\Si}$.

For fixed  $\Si$, $E$, and each hyperbolic $h_r$, 
\eqref{GL} has the 
following constant curvature (or magnetic field) solutions:
\begin{equation}
	\label{1.3}	{\psi\equiv 0,\quad a =a^{b_r}},\end{equation}
where 
$\psi$ is the zero-section on the 
line bundle $E$, and 
$a^{b_r}$ is a constant curvature connection 
satisfying
\begin{equation}
	\label{1.4}
	da^{b_r}= {b_r}\om_r,
\end{equation}
and 
{ \begin{equation}\label{b} {b_r}\equiv b(r,\Si,E):=\frac{2\pi \deg E}{\abs{\Si}_r}=\frac{b}{r},
\end{equation}
where $\deg E$ is the degree (or the first Chern number) of the bundle $E$ (see \secref{sec:2.2} for definitions), and 
\begin{equation}\label{b0}
	b:=\frac{2\pi\deg E}{\abs{\Si}}.
\end{equation}}

The value of $b_r$ in \eqref{b} is determined by the Chern-Weil correspondence, see \eqref{2.20} below.  {Once $\Si,\,E$ are fixed, the value $b_r$ can be computed explicitly using the Gauss-Bonnet formula \eqref{2.8} and the curvature parameter $r$ in the background metric \eqref{hr}. }


In the standard Ginzburg-Landau equations,  solutions of the form \eqref{1.3}
correspond to normal, non-superconducting
states.

For fixed $\Si$, $E$, let  $b_r=b/r$ with $r>0$.
The value $b_r$ turns out to be the smallest eigenvalue of $-\Lap_{a^{b_r}}$ (see \secref{secLap} for discussions).
 {We denote by 
\begin{align}\label{Kdef}
	K(r)=\Null(-\Lap_{a^{b_r}}-b_r)
\end{align} the finite dimensional null space  of $-\Lap_{a^{b_r}}-b_r$ acting on the space of square integrable sections on $E\to \Si$. Now, we define the  function 
\begin{equation}\label{beta}
		{\beta(r):=\min\Set{{\br{\abs{\xi}^4}}:\xi\in K(r),\br{\abs{\xi}^2}=1}.}
	\end{equation} }
Here and below, {$\br{f}:=\frac{1}{\abs{\Si}_r}\int f$. Note that, by H\"older's inequality, $\beta(r)>1$.} Eq.~\eqref{beta} 
 contains information about the energy of the solutions as a function of $r$ (see \corref{cor1.2}).

We define another function of \textit{threshold Ginzburg-Landau parameter}, $\kappa_c(r)$, as
\begin{equation}\label{kappac}
 {	\kappa_c(r):=\sqrt{\frac{1}{2}\del{1-\frac{1}{\beta(r)}}}.}
\end{equation}

{Finally, let $\cH^k$ and $\vec\cH^k$ be the Sobolev spaces of order $k$ of sections and weakly co-closed $1$-forms (i.e., $d^*\al =0$ in the distributional sense) on the line bundle $E\to\Si$,  $O_{\cH^k}$ and $O_{\vec \cH^k}$ stand for error terms in the sense of the norms in $\cH^k$ and $\vec \cH^k$, and let
\begin{align}\label{X}
		X^k:=\H^k\times \vec\H^k.
\end{align}
The definitions of these spaces are standard. See Appendix \ref{sec:A} for details. 
}


	
Now, we are ready to formulate our main results:

\begin{theorem}[Existence and uniqueness] \label{thm1.2}

	Let $(\Si,h_r),\,r>0$, be a non-compact Riemann surface equipped
	with the hyperbolic metric \eqref{hr} with 
	finite area.
	Let $E\to\Si$ be a unitary line bundle, with $b:=2\pi\deg E/\abs{\Si}>0.$	Suppose 
{	\begin{equation}\label{b0cond}
		{b\ne1/2},\quad
	\end{equation}
and $r>0$ satisfies}
	{\begin{align}
		& {0<\abs{\kappa^2r-{b}}\ll1, \quad  {(\kappa-\sqrt {b_r})(\kappa -\kappa_c (r))}>0},
		\label{1.5}\\
				\label{Kcond}
		&{\dim_\Cb K(r)=1.} 
		\end{align}}
Then, for each $r$ as above, there exists a solution 
	\begin{equation}
		\label{1.7'}
		(\psi(r), a(r)) 
	\end{equation}
	to \eqref{GL} in a neighbourhood of $U\subset X^k$, $k\ge2$,
	around $(0,a^{{b_r}})$. 

	Moreover, 
	the solution \eqref{1.7'} is unique in $U\subset X^k$ up to a  gauge symmetry (see \eqref{gaugetransf}).
\end{theorem}

We remark that our results in Section 7 enable us to establish the first, existence part of Theorem \ref{thm1.2} without imposing the non-degeneracy condition \eqref{Kcond}.

{\thmref{thm1.2} follows from the following: }
\begin{theorem}[Parametrization and asymptotics]
	\label{thm1.1}
		{Let \eqref{b0cond} hold and assume $\abs{\kappa^2r -b}\ll1$.} Then there exists $\eps>0$ s.th.~\eqref{GL} with metric \eqref{hr} has a $C^2$ branch of solutions {$(\psi_s,a_s,r_s)$, $s\in\Cb$, $\abs{s}\le \eps$,} 
%
%
	satisfying{\begin{align}
		&\psi_s=s \xi+O_{\cH^k}(\abs{s}^3),\label{1.8}\\
		&a_s=a^{{b_{r_s}}}+\abs{s}^2\al+O_{\vec \cH^k}(\abs{s}^4),\label{1.9}
	\end{align}
	where 
	$\xi=O_{\cH^k}(1)$ is gauge-equivalent to a holomorphic section of $E$ 
	corresponding to $a^{{b_{r_s}}}$, 
	 $\al=O_{\vec \cH^k}(1)$ is a co-closed $1$-form} and  satisfies, with
	{$*$ denoting the Hodge operator,}
	\begin{align}
		&d\al=\frac{1}{2}*\del{1-\abs{\xi}^2}.\label{1.11}
	\end{align}

{Moreover, if \eqref{1.5} and \eqref{Kcond} hold,  then we can take $s\in \Rb_{\ge0}$, and the solution $(\psi_s,a_s,r_s)$ is unique,} and the equation $r=r_s$ can be solved for $s$ to obtain
\begin{align}
			&{s}^2= {\frac{\kappa^2-b_r}{(\kappa^2-\tfrac{1}{2})\beta(r)+\tfrac{1}{2}}+O((\kappa^2-b_r)^2)}.\label{s} 
\end{align}

Furthermore, writing \eqref{s} as $s=s(r)$ gives, for any $r>0$ as above, the solution\begin{align}\label{1.7}
	(\psi(r),a(r))=(\psi_{s(r)}  ,a_{s(r)})  ,
\end{align}  as in  \thmref{thm1.2} (see \eqref{1.7'}).

\end{theorem}
{This theorem is proved in Sects. \ref{sec:4}--\ref{sec:5}.}

\begin{remark}\label{rem1.1}
{	Condition \eqref{1.5} guarantees the r.h.s. of \eqref{s}
	is positive, 
		and  was first isolated as a criterion
	for the existence of the Abrikosov vortex lattice in \cite{MR2560758,MR2904275,MR3123370,MR3758428}.} Through \eqref{b}, condition \eqref{1.5} gives rise to {
	the critical magnetic field $b_r=\kappa^2$.}
\end{remark}

\begin{remark}

The reason of condition \eqref{b0cond} is explained in \secref{sec:4}, {where we also show that an explicit bundle $E\to \Si$ satisfying condition \eqref{b0cond} is   
\begin{align}
	\label{1.6'}
	\Si=\Hb/\Ga(6),\quad \deg E=12.
\end{align}}
\end{remark}

\begin{remark}
	
		As we explain in Sections \ref{sec:2.3} and \ref{sec:2.5},
{		the number $\kappa^2r$ 
		corresponds to the average magnetic field in superconductors.
		Hence,
		condition \eqref{1.5} 
	gives an estimate of the neighbourhood of the critical average
	field strength we work in,} in terms of the topological degree $\deg E$,
	and the hyperbolic
	area of $\Si$ through the quantization relation \eqref{b}. 
	In this sense 	\eqref{1.5}  
	can be compared to Bradlow's condition
	for the existence of magnetic vortex on compact Riemann
	surfaces \cite{MR1086749}. 
\end{remark}

\begin{remark}\label{rem4}
{Similar results as ours 
have been obtained for compact Riemann surfaces in \cite{MR3790522,MR4049917,MR4213771,nagy2022nonminimal,MR1451329}.
It seems that ours is the first rigorous existence
theory for \eqref{GL} on non-compact Riemann surfaces. 
}
\end{remark}

{\begin{remark}
	It is possible to extend Thms.~\ref{thm1.2}--\ref{thm1.1} by dropping the second condition in \eqref{b0cond}, as we explain in \secref{secExt}.
\end{remark}}

{The proof of Theorem \ref{thm1.1} implies also the following proposition, whose proof can be found at the end of Section \ref{sec:3}:
\begin{proposition}
	\label{prop:loc-min}
{	The constant curvature solution  $(\psi\equiv 0, a =a^{b_r})$ (see \eqref{1.3}	) is a minimizer of the energy $\cE(\psi,a,h_r)$ if and only if
	\begin{equation}\label{b-norm}
		{b_r>\kappa^2. }
	\end{equation}}
\end{proposition}
%

We make two conjectures concerning the energy of the solution constructed in \thmref{thm1.1}.
\begin{conjecture}
	\label{conj:loc-min}
{	Under the assumption of \thmref{thm1.1}, solution \eqref{1.7} is a local minimizer of the energy $\cE(\psi,a,h_r)$ if and only if
	\begin{equation}\label{b-bif}
		{b_r<\kappa^2.}
	\end{equation}}
\end{conjecture}

We expect that  conjecture  above can be proven similarly to the corresponding result for the original  Ginzburg-Landau equations proven in \cite[Thm. 4]{MR2904275}.

 The significance of such a result is that it would show that by decreasing the curvature $b$, one passes from the (dynamically) stable constant curvature (normal) solution \eqref{1.3} to  the (dynamically) stable variable curvature solution \eqref{1.7}. 
 
 In fact, we expect stronger statements to be true:
 
\begin{conjecture}
	\label{conj:glob-min}
{The constant curvature solution  $(\psi\equiv 0, a =a ^{b_r})$ (see \eqref{1.3}	) is a global minimizer of the energy $\cE(\psi,a,h_r)$ if \eqref{b-norm} holds. 
	If \eqref{b-bif} holds, then solution \eqref{1.7} is a global minimizer of the energy $\cE(\psi,a,h_r)$.}
\end{conjecture}}

The energy associated to the constant curvature solution \eqref{1.3} is 
\begin{align}\label{GNE}
	\cE(0,a^{b_r},h_r)=\frac{1}{2}\del{{b_r^2}+\frac{\kappa^2}{2}}\abs{\Si}_r.
\end{align}
As a corollary of \thmref{thm1.1}, we obtain the following energy {expansion}:

\begin{corollary}\label{cor1.2}
	{
		Let conditions \eqref{b0cond}, \eqref{1.5}, \eqref{Kcond} hold as in \thmref{thm1.1}.}
	Then, for the solution $(\psi_{s(r)}, a_{s(r)})$ constructed in \eqref{1.7}, we have, 
	\begin{equation}
		\label{1.12}
		{{\cE(\psi_{s(r)},a_{s(r)},h_r)=\cE(0,a^{b_r},h_r)}{}}-{\frac{\abs{\Si}_r}4 \frac{\abs{\kappa^2-{b_r}}^2}{(\kappa^2-\tfrac{1}{2})\beta(r)+\tfrac{1}{2}}+O \del{\abs{\kappa^2-{b_r}}^3} }.
	\end{equation}


\end{corollary} 
This corollary is proved at the end of \secref{sec:5}.

{\begin{remark}
			{Since $(\kappa^2-\tfrac{1}{2})\beta(r)+\tfrac{1}{2}=\beta(r)(\kappa^2-\kappa_{c}^2(r))$, expansion \eqref{1.12} ensures that}
	$\cE(\psi_{s(r)},a_{s(r)},h_r)<\cE(0,a^{b_r},h_r)$, provided  	\begin{equation}\label{1.100}
		\kappa>\kappa_{c}(r).
	\end{equation} 
	Therefore, if \eqref{1.100} holds, then
	the solutions constructed in \thmref{thm1.1} are energetically favourable compared to the constant curvature one.

	Moreover, \corref{cor1.2} shows that there is an energy crossover between the trivial and nontrivial solutions at $\kappa_c(r)=\kappa$. 
	Thus, {at $\kappa=\kappa_c (r)$}, the gauge-translational symmetry (ensuring that $a$ has a constant curvature) is broken and a non-gauge-translational invariant ``ground state'' emerges at this point.
	
	Finally,  the function $\beta(r)\equiv \beta(r,\Si,E)$ yields the asymptotics of the GL energy of bundles over Riemann surfaces. 
\end{remark}}

\begin{remark}\label{remEquiv}
	{As was already indicated above,  in physics, $\psi$ is called, depending on the area,  either the order parameter or the Higgs field. In our context, it is represented by    the following two equivalent objects:}
\begin{enumerate}
	\item (Geometry) Sections of the unitary line bundle $E$
	with $n=\deg E\ne0$  over the surface $(\Si,h\equiv h_r)$ (as in this section);
	\item (Number theory) $\Ga$-automorphic functions  with weight $k=4\pi n/\abs{\Si}$
	and trivial multiplier system  {(see Appendix \ref{sec:2.4} and \cite{MR750670})}.
\end{enumerate}
Similar parallel can be drawn for the constant curvature connections
on $E$, weighted Maass operators on $\Si$, and gauge potentials with
constant magnetic fields with
strength $b$. {See \secref{sec:2.5} for details.}
\end{remark}

 \begin{remark}
 	Conceptually, {if one drops the second part in condition  \eqref{b0cond},} it is useful to introduce also the \textit{extended} Abrikosov function
 	$$\fullfunction{\beta}{\Null(-\Lap_{a^{b_r}}-b_r)\times \Rb_{>0}\times \Set{\text{Fuchsian groups}}}{\Rb_{\ge0}}{(\xi,r,\Ga)}{\norm{\xi}_{\cL^4}^4/\norm{\xi}_{\cL^2}^2}.$$
 	Then we have 
 	$\beta(r,\Ga)=\inf\Set{\beta(\xi,r,\Ga):\xi\in\Null(-\Lap_{a^{b_r}}-b_r)}$, c.f.~\eqref{beta}.
 	
 \end{remark}

\subsection{Organization {of the paper}}
{In \secref{sec:2}, after giving some preliminary definitions, we show that \eqref{GL} on $(\Si,h_r)$ with $h_r$ as in \eqref{hr} is equivalent to the rescaled Ginzburg-Landau equations, \eqref{2.13}, posed on $(\Si,h\equiv h_1)$.
Then the proof of \thmref{thm1.1} consists of two parts of analysis on the rescaled equation \eqref{2.13}. }

First, in \secref{sec:3}, 
we study the linearized problem associated to \eqref{GL},
which reduces to understanding  the spectral properties 
of the  Laplacian {$-\Lap_{a^b}$}
associated to a constant curvature connection
$a^b$, viewed as an operator acting  on square-integrable sections of the unitary line bundle $E\to \Si$. 
We show that the essential spectrum of $-\Lap_{a^b}$ {{is given by}
the half-line $[\tfrac{1}{4}+b^2,\infty)$,}
 and 
the lowest eigenvalue of this operator equals to $b$
whenever the space of cusp forms on $\Si$ is non-trivial, in which case we give {explicit description 
of $\Null(-\Lap_{a^b}-b)$.}
For precise statements, see \thmref{thm3.1}.

The operator $-\Lap_{a^b}$ is known in the physics literature as the magnetic 
Laplacian at constant field strength $b$, and 
is studied in e.g. \cite{MR776146,MR870891,MR1069475,Pnueli}.
For $b=0$, $-\Lap_{a^b}$ reduces to the Laplace-Beltrami operator
acting on the Poincar\'e half-plane, whose spectral properties are well
studied in \cite{MR1942691}.
For $b\ne0$, {eigenfunctions of $-\Lap_{a^b}$ are precisely } the weighted Maass forms in number theory. See e.g. \cite[{Sec. 2}]{MR1431508},
\cite{MR750670},
and the references therein.

Next, in \secref{sec:4}, we use Lyapunov-Schmidt reduction 
to show that a non-trivial branch of solution
of the form \eqref{1.8}--\eqref{1.9}
bifurcates from the constant curvature solutions 
\eqref{1.3}, provided the metric on $\Si$ 
satisfies  condition \eqref{1.5}. 
In \secref{sec:D}, we solve the key bifurcation equation \eqref{4.19}, which, by results from the previous section, amounts to solving \eqref{GL}. 
In \secref{sec:5}, we derive precise asymptotics for the solutions constructed in Sects. \ref{sec:4}--\ref{sec:D}. This proves \thmref{thm1.1} and \corref{cor1.2}.

{Lastly, in \secref{secExt}, we explain how to drop the non-degenerate condition in \eqref{b0cond} and extend the main results above to $\dim K>1$.}

\section{Preliminaries}\label{sec:2}
In this section, we explain the geometric setting
for the results and proofs in  this paper. 

In the {remainder} of this paper, the following geometric assumptions are 
always understood:
\begin{enumerate}
	\item The underlying surface $\Si$ is of the 
	form \eqref{2.1},  with finite area, $g$ genus, 
	$m$ cusp, and no elliptic points (e.g.
	the principal congruence subgroup $\Si=\Hb/\Ga(N)$ with $N\ge2$, defined in  \eqref{2.4} below);
	\item {$b>0$ in \eqref{b} (which can always be achieved by changing orientation so that $\deg E=1,2,\ldots$).} 

\end{enumerate}

\subsection{Non-compact Riemann Surfaces}\label{sec:2.1}
Let $\Si$ be a connected Riemann surface. The Uniformization Theorem states
that if $\Si$ is non-compact and not flat, then 
\begin{equation}\label{2.1}
	\Si\cong \Hb/\Ga,
\end{equation}
where $\Hb$ is the 
Poincar\'e half-plane, 
$$\Hb:=\Set{z:z\in\Cb,\,\Im z>0},$$ and $\Ga$ is a Fuchsian group,
i.e. a discrete subgroup of 
$$PSL(2,\Rb):=SL(2,\Rb)/\Set{\pm \mathbf{1}}$$ 
acting freely on $\Hb$.
Here, the action of $SL(2,\Rb)$ on $\Hb$ is by M\"obius transform,
\begin{equation}
	\label{2.2}
	\gamma z=\frac{az+b}{cz+d}\quad \del{\gamma=\begin{pmatrix}a&b\\c&d\end{pmatrix}\in SL(2,\Rb)}.
\end{equation}
By convention, we define $\g\infty=a/c$.

An important class of examples are the Riemann surfaces 
\begin{equation}\label{2.3}
	\Si:=\Hb/\Ga(N), \quad N=1,2,\ldots,
\end{equation}
where $\Ga(N)$ is the principal congruence subgroup of level $N$,
\begin{equation}
	\label{2.4}
	\Ga(N):=\Set{\gamma=\begin{pmatrix}a&b\\c&d\end{pmatrix}\in SL(2,\Zb):a\equiv d\equiv 1, b\equiv c \equiv 0 \mod N}.
\end{equation}
By definition, $\Ga(N)$ is a normal subgroup of the modular group $SL(2,\Zb)$ for each $N$.

\begin{definition}[cusp]\label{defn:cusp}
	Let $\Ga$ be a Fuchsian group.
A point $c\in\Rb\cup\Set{\infty}$ is called a \textit{cusp}
of $\Ga$ if and only if 
there is an element $\g\in \Ga$ that is conjugate-equivalent
to some horizontal translation $z\mapsto z+h,\,h\in \Rb$, s.th. $\g c=c$.
\end{definition}

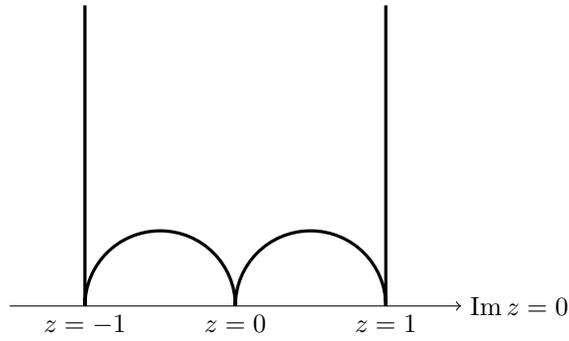
\begin{figure}[h]
	\centering
	\begin{tikzpicture}[scale=2]
		\draw [->] (-1.5,0)--(1.5,0);
		\node [right] at (1.5,0) {$\Im z=0$};
		
		\draw[black,very thick] (-1,2)--(-1,0) to [out=90,in=180] (-0.5,.5) to [out=0,in=90](0,0) to [out=90,in=180] (0.5,.5) to [out=0,in=90](1,0)--(1,2);
		
		\node [below] at (0,0) {$z=0$};
		\node [below] at (-1,0) {$z=-1$};
		\node [below] at (1,0) {$z=1$};
	\end{tikzpicture}
	\caption{A fundamental domain of $\Gamma(2)$ in $\mathbb{H}$  with three cusps $c_1=1,c_2=0,c_3=\infty$. }
	\label{fig:fdofga2}
\end{figure}

For example, if $\Ga=SL(2,\Zb)$, then the only cusp of $\Ga$ is $c=\infty$,
as every integral translation $z\mapsto z+n,\,n\in\Zb$ fixes $c$. 
If $\Ga=\Ga(2)$, then there are three cusps with $c_1=0,c_2=1, c_3=\infty$. See \figref{fig:fdofga2}.  (Note that $-1$ is equivalent to $c_2$ through translation $z\mapsto z+2$.)

Denote by $m$ the number of distinct cusps and {$g$ the number of genera of $\Ga$.}
For  the principal congruence subgroup $\Ga=\Ga(N)$  defined in  \eqref{2.4},
the classical results from \cite[{Sec. 1.6}]{MR0314766}
show that 
\begin{align}
	\label{2.5}
	m_N=&\left\{\begin{aligned}
	&3\quad (N=2),\\
	&\frac 1 2 N^2 \prod_{p \mid N}\left(1 - \frac 1 {p^2}\right)\quad 
	(N\ge 3),
\end{aligned}\right.\\
\label{2.6}
g_N=&1+\frac{N-6}{12}m_N.
\end{align}

\subsection{Metric and Rescaling}\label{sec:2.3}
Let $r>0$.
As in \secref{sec:1},
we equip the space $\Hb$ with the following
families of {hyperbolic} metrics and induced area $2$-forms:
\begin{align}
	\label{2.9}
	h_r &= \frac{r}{(\Im z)^2} dz \otimes d\bar{z},\\
	\label{2.10}
	\om_r&=\frac{ri}{2\del{\Im z}^2}dz\wedge d\bar z.
\end{align}
These carry over to Riemann surfaces of the form $\Si=\Hb/\Ga$.
The surface $(\Si,h_r)$ has constant curvature $-1/r$
and surface area $\abs{\Si}_r=r \abs{\Si}$,
where $\abs{\Si}$ denotes the area of $\Si$ w.r.t. the standard hyperbolic $\om\equiv \om_r\vert_{r=1}$.
 
{Suppose $(\Si,h\equiv h_1)$ has finite area,   $g$ genus,   $m$ cusps,  and no elliptic points. Then 
Gauss-Bonnet formula gives}
\begin{equation}
	\label{2.8}
	\abs{\Si}=2\pi(2g-2+m),
\end{equation}
See e.g. \cite[{p.43, Eqn. (2.7)}]{MR1942691}.
For example, if   $\Si_N=\Hb/\Ga(N)$ with $N\ge2$, then $\Si_N$ has no elliptic point. Hence, by \eqref{2.5}--\eqref{2.6} and Gauss-Bonnet formula, we have
\begin{equation}
	\label{2.8N}
	\abs{\Si_N}=\frac{\pi Nm_N}{3}.
\end{equation}


{For $h=h_r$ given by \eqref{2.9}, the Ginzburg-Landau energy functional \eqref{1.1}
reduces to
\begin{equation}\label{2.12}
	\cE_r(\psi,a)=\int_\Si \del{\frac{1}{2}\abs{\grad_a\psi}^2+\frac{1}{2}\abs{da }^2+\frac{\kappa^2}{4}\del{\abs{\psi}^2-r}^2}\,\om_1,
\end{equation}
{defined on the base surface $(\Si,h\equiv h_1)$.}
To see this, we distinguish the quantities related to the metric $h_r$ by tildes and consider 
	\begin{align}\label{scale1}
		\tilde \psi = r^{-1/2}\psi,\quad  \tilde a = r^{-1/2} a,
	\end{align} 
	Using the relations $\tilde \grad =r^{-1/2}\grad$ and $\tilde d = r^{-1/2}d$ (see e.g.~\cite[{Sect. B.1}]{MR4049917}), we have
	\begin{align}\label{scale2}
		{\grad _{\tilde a}\tilde \psi}=\frac1r\grad _a \psi, \quad \tilde d\tilde  a =\frac1rda,\quad \abs{\tilde\psi}^2=\frac1r\abs{\psi}^2.
	\end{align}
	Plugging \eqref{scale2} into \eqref{1.1} and using $\om_r=r\om_1$ (see \eqref{2.10}), we find
	\begin{align}
	r	\cE(\tilde \psi,\tilde a,h_r)  =&\int_\Si \del{\frac{1}{2}\abs{\grad _{\tilde a}\tilde \psi}^2+\frac{1}{2}\abs{\tilde d\tilde  a}^2+\frac{\kappa^2}{4}\del{\abs{\tilde\psi}^2-1}^2} r^2\om_1 \notag \\
		=&\int_\Si \del{\frac{1}{2}\abs{\grad_a\psi}^2+\frac{1}{2}\abs{da}^2+\frac{\kappa^2}{4}\del{\abs{\psi}^2-r}^2} \om_1=\cE_r(\tilde \psi,\tilde a),\label{enRel}
	\end{align}
	which gives \eqref{2.12}.
}

The Euler-Lagrange equations for $\cE_r(\psi,a)$ are the rescaled
\eqref{GL}, given by
\begin{equation}
	\label{2.13}
	\begin{aligned}
		-\Lap_a\psi&=\kappa^2\del{r-\abs{\psi}^2}\psi,\\
		d^*d a&=\Im (\bar{\psi}\grad_a\psi),
	\end{aligned}
\end{equation}
{in the space $X^k,\,k\ge2$ defined in \eqref{X}.}
{For suitable values of $r>0$, we seek solution pair $(\psi,a)$ to \eqref{2.13} on the unscaled surface $(\Si,h)$.}
In \secref{sec:2.5}, 
we show that 
the  parameter $r>0$ 
corresponds
to the average field {(magnetic flux)} strength.
{In Type II superconductors,} the variation of the latter 
triggers 
second-order phase transition.
Hence, in what follows we will use
$r$ as the bifurcation parameter.

{The rescaled Ginzburg-Landau equations
\eqref{2.13} are the central objects of the
subsequent sections.}

\subsection{Quantization of Magnetic Flux}\label{sec:2.5}
In this subsection, we state the extension of the Chern-Weil
correspondence (known in the physics literature as magnetic flux quantization) to non-compact Riemann surfaces.

Recall that $\deg E$ denotes the topological degree,
or the first Chern number, of the unitary line bundle $E$, see Appendix \ref{sec:A}.

\begin{theorem}[Chern-Weil correspondence for non-compact Riemann surfaces]\label{thm2.1} 
	Let $\Si$ be a non-compact Riemann surface,
	and $E\to \Si$ be a unitary line bundle.
	For every connection $a$ on $E\to\Si$ with $\abs{\int_\Si da}<\infty$,
	there holds
	\begin{equation} 
		\label{2.19}  
		\frac{1}{2\pi}\int_\Si  da =   \deg E.
	\end{equation}
\end{theorem} 
The proof of this theorem is given in Appendix \ref{sec:flux-quant-pf}.

{Since $\deg E\in\Zb$, equation \eqref{2.19} implies quantization of the average magnetic field (magnetic flux). 

{The Chern-Weil correspondence imposes a direct constrain on constant curvature connections. Indeed, 
let $a^\beta$ be a constant curvature connection on $(\Si,h_r)$ satisfying $da^\beta=\beta\om_r$ for some $\beta\in\Rb$ (e.g. \eqref{1.4}).
Then, by relation \eqref{2.19}, we have
\begin{align}  \label{2.20}
	\beta= \frac{1}{|\Si|_r}\int_\Si  da  = \frac{2\pi \deg E}{|\Si|_r},
\end{align}
which, together with the relation $\abs{\Si}_r=r\abs{\Si}$, gives  $\beta=b_r$ as in \eqref{b}.}

\eqref{2.20} relates the average magnetic field (curvature) on $\Si$ to the geometry of $\Si$,
and the topology of the line bundle $E\to \Si$.
{Indeed,
\eqref{2.20} shows that varying  the metric on $\Si$
as in \eqref{2.9}--\eqref{2.10}
amounts to varying the
constant curvature solution.}


\section{Linearized Ginzburg-Landau Equations}\label{sec:3}

From now on, the central object of study will be the rescaled GL equations \eqref{2.13}.  
	As explained in \secref{sec:2.3}, these equations are posed on the unscaled surface $(\Si,h\equiv h_1)$. In this section, we consider the linear problem 
	associated to \eqref{2.13}.

The Ginzburg-Landau equations \eqref{2.13} have constant curvature solutions
\begin{equation}
	\label{3.1}
	\text{$(\psi,a)=(0,a^{b})$, where  $a^{b}$ is any constant curvature connection with $da^{b}={b}\om$.}
\end{equation}
{As a consequence of the Chern-Weil correspondence \eqref{2.19}, since $a^b$ is a constant curvature connection on $(\Si,h \equiv h_1)$, the number $b$ is given by \eqref{b0}.}

Linearizing \eqref{2.13} at the solution \eqref{3.1}, we get a decoupled system
\begin{align}
	\label{3.2}
	(-\Lap_{a^b}-\kappa^2r)\phi&=0,\\
	\label{3.3}
	d^*d\alpha&=0.
\end{align}
These are the main objects of this section. 
At this point, the unknown in \eqref{3.2}--\eqref{3.3} here is a (section, connection)-pair on a unitary line bundle $E\to (\Si,h_1)$. The curvature parameter $r$ from the unscaled GL equations now enters only through the second term in the l.h.s. of \eqref{3.2}.  


Our goal now is to obtain an explicit description of the solutions
to \eqref{3.2}--\eqref{3.3} in the Sobolev space $X^s$ defined in \eqref{X}.

{
	\subsection{Solving the Maxwell Equation}
	First, we solve the homogeneous equation \eqref{3.3} (the free Maxwell equation) for the connection $a$.  
	
	By \cite[{Lem. 3.2}]{MR4049917}, $\al$ solves \eqref{3.3} if and only
	if $\al$ is a constant curvature connection. Thus $a^b$ is a solution. By the Chern-Weil correspondence, 
	the linearization, \eqref{3.3}, around $(\psi,a)=(0, a^b)$ must be satisfied by a $1$-form $\al$ of degree 0.
	So again, by \cite[{Lem. 3.2}]{MR4049917},  $\al$ must be a flat connection: In the distributional sense,
	\begin{equation}
		\label{3.4}
		d^*d\al=0 \iff \al \text{ is flat } \iff d \al = 0.
	\end{equation}
	So $\al$ must be closed; but by the definition of \eqref{X} it must also be co-closed and hence harmonic. This establishes the following.
	\begin{proposition}
		\label{prop3.1}
		$d^*d\ge0$ and the solution space to \eqref{3.3} in $\vec \cH^2$ is 
		\begin{equation}
			\label{3.5}
			\Om:=\Null d^*d\vert_{\vec \cH^2}=\Set{\text{harmonic 1-forms on}\,\,\, {\Si}}.
		\end{equation}
	\end{proposition}
}

\subsection{The Magnetic Laplacian}\label{secLap}

{Recall that $b=2\pi\deg E/\abs{\Si}$, and we choose the orientation of $E$ that makes   $b>0$. }
By the classification of constant curvature connections, \thmref{thmB.2},
an one-form  $\beta$ on $\Si$ satisfies  $d\beta=b\om$ and the equivariant 
condition  if and only if $\beta$ is gauge-equivalent to
\begin{equation}
	\label{3.6}
	a^b:=by^{-1}dx.
\end{equation}
Thus in what follows we fix the canonical choice $a^b$ as in \eqref{3.6}.

In the remaining subsections, we study the spectral properties of 
the magnetic Laplacian $-\Lap_{a^b}$ in the l.h.s. of \eqref{3.2}.
Locally, in the rectangular coordinate $z=x+iy$, we have
		\begin{align}\label{3.7} -\Delta_{a^b}=-y^2(\di_{xx}+\di_{yy}) +2i b y\p_x+b^2\text{ acting on }{F_\Si\subset \Hb}.
		\end{align} 
This follows from {direct} computation using the standard definitions in \secref{sec:A}.

We prove the following.
\begin{theorem}\label{thm3.1} 
	Let $\Si=\Hb/\Ga$ be a non-compact Riemann surface with $m$ cusps and no elliptic points. 
	\begin{enumerate}[label=(\alph*)]
		\item $- \Delta_{a^b}$ is self-adjoint; 
		\item $- \Delta_{a^b}\ge b$; 
		\item Let $\cS(\Si)\equiv \cS_k(\Si)$ denote the space of cusp forms on $\Si$ with weight $k=2b=4\pi n/\abs{\Si}$. Then $b$ is an eigenvalue of  $- \Delta_{a^b}$ if and only if $\cS(\Si)\neq \emptyset$, and the multiplicity of $b$ equals to $\dim \cS(\Si)$; 
		\item {The essential spectrum of $- \Delta_{a^b}$
		 consists of $m$ branches each of which filling  in the semi-axis $[1/4+b^2, \infty)$.  Hence,
\begin{equation}\label{essEq}
	\s_{\rm ess} (-\Lap_{a^b})= [1/4+b^2, \infty).
\end{equation}}
	\end{enumerate}
\end{theorem}
\begin{remark}
	If the field strength $b=0$, {then $E$ is the trivial line bundle,} and $-\Lap_{a^b}$ reduces to
	the hyperbolic Laplacian $-\Delta$ on $\Si$.
	The spectral properties of this operator have important bearings for number theory. 
	See e.g. \cite{MR1942691} and the references {therein}.

	For $b>0$, some of the results above have been obtained in \cite{MR776146, MR870891,MR1069475,Pnueli}
	for specific choices of $\Si$ in relevant physics context.
	For the connection to weighted Maass forms in number theory, see \cite[{Sec. 2}]{MR1431508}.
\end{remark}

{\begin{remark}\label{remEmb}
	For $b=1/2$, {the eigenvalue $b$ is embedded at the bottom of the essential spectrum of $-\Lap_{a^b}$.} 
\end{remark}}

\begin{proof}[Proof of \thmref{thm3.1}]
	The proof of Part (a) is standard
	and can be found in \cite[{Sec. 2.1}]{MR1431508}.
	 Parts (b)-(c) follow from the Weitzenb\"ock formula given in \secref{sec:3.3}. Part (d) follows from the results
in \secref{sec:3.4}.
\end{proof}

We also obtain a  description
of the solution space to the static Schr\"odinger equation \eqref{3.2}. Before we state this result, we need some preliminary definitions. 

Let $\Ga\subset PSL(2,\Rb)$ be a Fuchsian group. Let $c_i$ be a cusp of $\Ga$. Then the stabilizer of $c_i$ is an infinite cyclic group generated by some parabolic transform. In symbols, 
$$Stab(c_i,\Ga)\equiv \Set{\gamma\in \Ga:\g c_i=c_i}=\left\langle\begin{pmatrix}1&r_i\\0&1\end{pmatrix}\right\rangle \text{ for some }r_i\in\Rb.$$
We call $\g_i\in SL(2,\Rb)$ {\it a scaling matrix} of $c_i$ if $\g_ic_i = \infty$ and $\g_i\del{\begin{smallmatrix}1&r_i\\0&1\end{smallmatrix}}\g_i^{-1}=\del{\begin{smallmatrix}1&1\\0&1\end{smallmatrix}}.$ Such $\g_i$'s exist, and are unique up to translation. See \cite[Chapt. 2.2]{MR1942691} for details.


\subsection{Weitzenb\"ock-type formula} \label{sec:3.3}

In this subsection, let $a$ be an arbitrary 
unitary connection on a smooth complex line bundle $E\to\Si$ (see Appendix \ref{sec:A} for the definitions). 

We decompose the covariant derivative  $\nabla_a$ into $(1, 0)$ and $(0, 1)$ parts as $\nabla_a=  \partial_a' +  {\partial_a''}$, where
\begin{align} \label{pA}\partial_a' := \partial +  a_c,\  \partial_a'' :=\overline{\partial} +  \bar a_c.\end{align} 
Here $\p := \frac{\partial}{\p z}\otimes d z$ and $\overline{\partial} := \overline{\frac{\partial}{\p z}}\otimes d\overline{z}$, where, as usual, $ \frac{\partial}{\p z}\equiv  \p_{z}:=(\partial_{x_1} - i \partial_{x_2})/2$,  $ \overline{ \frac{\partial}{\p z}} \equiv  \p_{\bar z}:=(\partial_{x_1} + i \partial_{x_2})/2$, and 
\begin{align} \label{compl-conn} a_c:=\frac12 (a_1 - ia_2) \otimes d z,\ \bar a_c:=\frac12 (a_1 + ia_2) \otimes d \bar z. \end{align}

With the definitions above, we can rewrite \eqref{pA} as
\begin{align} \label{pA''}
	\partial_a'' = \tilde\p_{a}''\otimes d\overline{z} ,\ \quad  \tilde\p_{a}'' := \bar\p_z +  \bar a^c,\ \quad a^c:= a_1- i a_2.\end{align} 
{Throughout this subsection, $a^c$ denotes the complexification of a real valued $1$-form, and is not to be confused with the constant curvature connection $a^b$.}

In the reverse direction, we have $a = 2\re a_c$ and  
\begin{equation}
	\tilde\p_{a}''  := \frac12 (\nabla_{1} + i \nabla_{2}). 
\end{equation}

In terms of $a_c$, the curvature is given by   $F_a=2 \re \bar \p a_c$. 
Now we prove the following relations:

\begin{proposition}\label{prop:Delta_A-repr}
	We have
	\begin{align}
		\label{curv-formula} * F_a &= \frac12 [\tilde\p_{a}'', \tilde\p_{a}''^*],\\
		\label{Lapl-formula} -\Delta_a &=  {\partial''_a}^*{\partial''_a} + *F_a,\\
		\label{Lapl-ineq} -\Delta_a &\ge *F_a. 
	\end{align}
\end{proposition} 
\begin{proof}
	We compute in local coordinates. Using the relations $\n_{1} =\frac12 (-\tilde\p_{a}''+\tilde\p_{a}''^*)$, $\n_{2} =\frac{1}{2i} (\tilde\p_{a}''+\tilde\p_{a}''^*)$, together with the
	 expression \eqref{curv-expr} for the curvature $F_a$, we compute 
	\begin{align}\label{barp-comm}    &F_a =-\frac14 [\tilde\p_{a}''-\tilde\p_{a}''^*, \tilde\p_{a}''+\tilde\p_{a}''^*]= \frac12 [\tilde\p_{a}'', \tilde\p_{a}''^*],
	\end{align} 
	which gives \eqref{curv-formula}. 
	
	To find the expression for $\Delta_a$, we use the relation $\Delta_a=\n_i \n_{i}$ to compute 
	\begin{equation}\begin{aligned}
			\Delta_a &=\frac14 (\tilde\p_{a}''-\tilde\p_{a}''^*)^2  - \frac14 (\tilde\p_{a}''+\tilde\p_{a}''^*)^2\\&=-\frac12 (\tilde\p_{a}'' \tilde\p_{a}''^*+\tilde\p_{a}''^* \tilde\p_{a}'')\\
			& =-\tilde\p_{a}''^* \tilde\p_{a}''- \frac12 [\tilde\p_{a}'', \tilde\p_{a}''^*],
		\end{aligned}
	\end{equation}
	which gives \eqref{Lapl-formula}.
	
	Since ${\partial''_a}^*{\partial''_a}\ge 0$, eq. \eqref{Lapl-formula} implies \eqref{Lapl-ineq}.
 \end{proof}

\propref{prop:Delta_A-repr} and the fact that $*F_{a^b}=b$ imply the next two corresponding relations in the constant curvature case: 
\begin{corollary}\label{cor:Delta_A-part_A} Let $a^b$ be a constant curvature connection
	s.th. $*F_a=b$ Then 
	\begin{align}\label{Lapl-b-ineq} -\Delta_{a^b} &\ge b,\\
	\label{3.9}	K:=\Null (-\Delta_{a^b} - b) &= \Null {\partial''_{a^b}}.
	\end{align} 
	Hence,  $b$ is an eigenvalue of  
	$-\Delta_{a^b}$ if and only if $\Null {\partial''_{a^b}}\neq \{0\}$. \end{corollary}


\begin{proof}[Proof of  Theorem \ref{thm3.1}, Parts (b), (c)]  Part (b) follows from \eqref{Lapl-b-ineq}. To complete the proof of  Part (c), we now claim
	\begin{align} \label{Nulldbar}\Null {\partial''_{a^b} }\vert_{\Om^0(E)}=H^0(E),
	\end{align}
where $\Om^0(E)$ denotes the space of sections of the line bundle $E$, and $H^0(E)$ the space of
holomorphic sections.

Indeed, suppose \eqref{Nulldbar} holds. 
By the equivalence described at the end of \secref{sec:1.1},
the space $H^0(E)$ of 
holomorphic sections of a unitary line bundle $E$ 
are  isomorphic (as {a} complex vector space) to the space $M_k(\Si)$ of
modular forms on $\Si$,
with weight $k=2b=4\pi\deg E /\abs{\Si}$.
Since $\Si$ is non-compact, 
the intersection $\cL^2(\Si)\cap M_k(\Si)$ equals to the space of cusp forms $\cS_k(\Si)$
with weight $k$. This fact follows from the results in  \cite[{Sect. 3}]{MR1942691} 
 and \cite[{Sects. 3-4}]{MR2112196}. 
Since we seek $H^s$ solution with $s\ge0$, the desired result follows from here. 

It remains to prove \eqref{Nulldbar}.
 On  bundles over Riemann surfaces, we have that  
	\begin{align} \label{dbarsqrd}
		F^{0,2}_a = \partial''_a  \circ \partial''_a  = 0. 
	\end{align}
	For higher dimensional manifolds, condition (\ref{dbarsqrd}) provides the means to construct a canonical holomorphization of $E$. The precise result is the theorem below, which gives \eqref{Nulldbar}, whereby completing the proof  of  Theorem \ref{thm3.1}(c). \end{proof}

{The following result was obtained in \cite{MR1079726, MR4049917}.  \cite{MR4049917} develops a more constructive approach when
$M$ is a closed Riemann surface.}
\begin{theorem} \label{thm:Holo}
	Given a $C^\infty$ complex vector bundle $E$ over $M$ with connection $\grad_a$ such that $\partial''_a  \circ \partial''_a  = 0$,
 there is a unique holomorphic vector bundle structure on $E$ such that $\partial''_a = \partial''$.
\end{theorem}

\subsection{Essential Spectrum of the Magnetic Laplacian} \label{sec:3.4}

In this subsection, we compute the essential spectrum of the  linearized operator
$-\Lap_{a^b}$.
Here we use a direct method of geometric decomposition, and 
we will use freely the standard definitions from  Appendix \ref{sec:A} and results from Appendix \ref{sec:aut-fact-ccc}.


	First,  we identify $\Si$ with a fundamental domain $F_\Si\subset \Hb$ of $\G$,  with the sides identified  as in the proof of Theorem \ref{thm2.1}. Let
	\begin{equation}\label{gammai}
		\g_i:=\left\{\begin{aligned}
			&\begin{pmatrix}0&-1\\1&-c_i\end{pmatrix}\quad (c_i\ne \infty),\\
			&\begin{pmatrix}1&0\\0&1\end{pmatrix} \quad (c_i=\infty).
		\end{aligned}\right.
	\end{equation}
	Then $\g_i\in SL(2,\Rb)$ and 
	$\gamma_ic_i=\infty$, where the action of $\g_i$ on {$F_\Si\subset \Hb$ }is the M\"obius transform \eqref{2.2}. Denote by $\varphi_i$ the action of $\g_i$ on $F_\Si\subset \Hb$. Then
	\begin{equation}\label{varphi}
		\varphi_i:\left\{\begin{aligned}
			&z\mapsto -\frac{1}{z-c_i}\quad (c_i\ne \infty),\\
			&z\mapsto z \quad (c_i=\infty).
		\end{aligned}\right.
	\end{equation}
See \figref{fig:phi} below.
	
	Now, we decompose  $F_\Si$  
	into a compact connected set, $U_0$, and  neighbourhoods $U_i$ of the cusps $c_i$, in
	such a way that
\begin{enumerate}
	\item   $U_i\cap U_j=\emptyset$ for $1\le i\ne j\le m$, and $U_0:=F_\Si\setminus\bigcup_{i=1}^m U_i$ is compact; 
	\item   For $i=1,\ldots,m$, 
	  the map $\varphi_i$ maps the domains  $U_i$ isometrically onto the half-cylinder 
	 \begin{align}\label{Xc} Z_i:=\{z\in \C: \Im z>s_i\}/\Z,
	 \end{align}  for some $s_i\gg 1 $.
\end{enumerate}
So long as $s_i\gg1$,  such decomposition of $F_\Si$ is easy to construct, see \figref{fig:decomp} below.
	
	\begin{figure}[h]
		\centering
	\begin{tikzpicture}[scale=1.6]
	\draw [->] (-1.5,0)--(1.5,0);
	\node [right] at (1.5,0) {$\Im z=0$};
	\node [below] at (0,0) {$z=0$};
	\node [below] at (-1,0) {$z=-1$};
	\node [below] at (1,0) {$z=1$};
	
	\draw [->] (-1.3,.25)--(-1.1,.25);
	\draw [->] (-1.3,.25) to [out=55,in=150](.8,.55);
	\draw [->] (-1.5,0)--(1.5,0);
	\draw [->] (-1.5,0)--(1.5,0);
	\draw [->] (-1.5,0)--(1.5,0);
	\draw [->] (-1.5,0)--(1.5,0);
	\node [left] at (-1.3,.25) {$U_1$};
	
	\draw [->] (-1.3,1.25)--(-1.1,1.25);
	\node [left] at (-1.3,1.25) {$U_3$};
	
	\draw [->] (-.5,.25)--(-.2,.25);
	\node [left] at (-.5,.25) {$U_2$};

	\draw [->] (1.3,.75)--(1.1,.75);
	\node [right] at (1.3,.75) {$U_0$};	
	
	\draw[black,thick] (-1,2)--(-1,0) to [out=90,in=180] (-0.5,.5) to [out=0,in=90](0,0) to [out=90,in=180] (0.5,.5) to [out=0,in=90](1,0)--(1,2);
	
	\clip (-1,2)--(-1,0) to [out=90,in=180] (-0.5,.5) to [out=0,in=90](0,0) to [out=90,in=180] (0.5,.5) to [out=0,in=90](1,0)--(1,2);

	\fill[fill=gray] (.5,0) arc (0:180:.5);
	\fill[fill=gray] (1.5,0) arc (0:180:.5);
	\fill[fill=gray] (-.5,0) arc (0:180:.5);
	\fill[fill=gray] (-1,1) rectangle (1,2);

\end{tikzpicture}
\caption{Schematic diagram for the decomposition of a fundamental domain of $\Gamma(2)$ in $\mathbb{H}$ with three cusps $c_1=1,c_2=0,c_3=\infty$. }
		\label{fig:decomp}
	\end{figure}
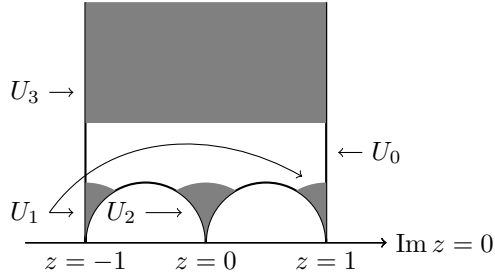

	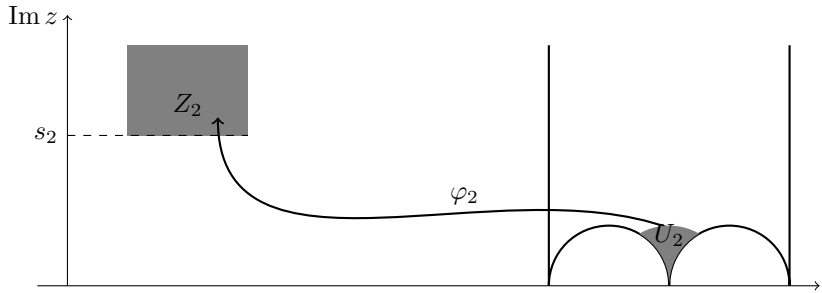
\begin{figure}[h]
	\centering
	\begin{tikzpicture}[scale=1.6]
		\draw [->] (-5.25,0)--(1.25,0);
		
		\draw [->] (-5,0)--(-5,2.25);
		\node [left] at (-5,2.25) {$\Im z$};

		%
		
		
		\fill[fill=gray] (-4.5,1.25) rectangle (-3.5,2);
		\node at (-4,1.5) {$Z_2$};
		\draw[dashed] (-5,1.25) to (-3.5,1.25);
		\node[left] at (-5,1.25) {$s_2$};
		
		\draw [->,thick] (0,.48) to [out=160, in=-90](-3.75,1.4);
		\node [above] at (-1.7,.6) {$\varphi_2$};
		
		\draw[black,thick] (-1,2)--(-1,0) to [out=90,in=180] (-0.5,.5) to [out=0,in=90](0,0) to [out=90,in=180] (0.5,.5) to [out=0,in=90](1,0)--(1,2);
		
		\clip (-1,2)--(-1,0) to [out=90,in=180] (-0.5,.5) to [out=0,in=90](0,0) to [out=90,in=180] (0.5,.5) to [out=0,in=90](1,0)--(1,2);

		\fill[fill=gray] (.5,0) arc (0:180:.5);
		\node[above] at (0,.25) {$U_2$};

	\end{tikzpicture}
	\caption{Schematic diagram illustrating the map $\varphi_2$ from \eqref{varphi} associated to cusps $c_2=0$. }
	\label{fig:phi}
\end{figure}

	 In the next two lemmas, we analyze the spectral property of $-\Lap_{a^b}$
in each domain  $U_i$ separately. 
\begin{lemma}\label{lem:ess-spec-Lapl1} Let $-\Delta_{a^b}\big|_{U_i}$ be the restriction of $-\Delta_{a^b}$ on $L^2(U_i)$ with the Dirichlet boundary conditions on $\p U_i$. Then 
	\begin{align}\label{ess-spec-Lapl1} \s_{\rm ess}(-\Delta_{a^b})=\bigcup_{i=1}^m \s_{\rm ess}(-\Delta_{a^b}\big|_{U_i}).\end{align}   
\end{lemma} 
\begin{proof} 
	1. First, note that the compact domain $U_0$ and the cuts do not contribute to the essential spectrum. Next,	let $\Set{U_0',\ldots,U_m'}$ be an open covering $\Si$,
	such that 
	\begin{enumerate}
		\item $U_i'\supset U_i$ for every $i=0,\ldots,m$;
		\item  $U_i'\setminus U_i$ is bounded in $F_\Si$ for every $i=0,\ldots,m$;
		\item $U_i'\cap U_j'=\emptyset$  for $1\le i\ne j\le m$.
	\end{enumerate}
For Item 3 above, we note that for sufficiently small $U_i,\,i=1,\ldots m$, we can always choose a slightly larger $U_i'\supset U_i $ s.th. $U_i'$ are still mutually disjoint for $1\le i\ne j\le m$ (see \figref{fig:decomp}).

	Let $\chi_i$ be a partition of unity on $\Si$ adapted to $U_i'$,
	such that
	\begin{align}
		\sum_{i=0}^m\chi_i^2&=\one,\\
		{\chi_i\vert_{U_i}\equiv 1},\quad \supp\chi_i&\subset  {U_i'}\quad (i=0,\ldots,m).\label{chi2}
	\end{align} 
	
	Then the IMS formula (see \cite{MR883643}) gives
	\begin{align}\label{IMS-form}  - \Delta_{a^b}&=\sum_{i=1}^m T_i +R,\ \quad T_i:= \chi_i(-\Delta_{a^b})\chi_i,\\
		\label{R}  R:&=\chi_0(-\Delta_{a^b})\chi_0 -\sum_{i=0}^m|\n\chi_i|^2.
	\end{align}

	2. Since $R$ is localized in a compact domain
	$$\supp \chi_0\cup \supp \grad \chi_1\cup\cdots \cup \supp \grad\chi_m,$$ we expect that  $-\Lap_{a^b}$
	and $\sum T_i$ have the same essential spectrum.  To show this, we will repeatedly use the fact that compact operators form a two-sided ideal among the bounded operators.
	
	By \thmref{thm3.1}, Part (b), the spectrum $\si(-\Lap_{a^b})\subset \Rb_{\ge0}$. Hence $(-\Lap_{a^b}+1)^{-1}$ is well-defined and bounded on $\cL^2$. 
	
	Claim: $R(-\Lap_{a^b}+1)^{-2}$ is a compact operator.

	Proof of the claim: Since for each $i=1,\ldots,m$,  the set $\supp \abs{\grad\chi_i}$ is bounded in $F_\Si$, it follows that the operators $|\n\chi_i|^2(-\Lap_{a^b}+1)^{-2}$ are all compact. Now, write
	\begin{align*}
		&\Lap_{a^b}\chi_0(-\Lap_{a^b}+1)^{-2}=BK,\\ &B:=\Lap_{a^b}(-\Lap_{a^b}+1)^{-1},\\ &K:=\chi_0(-\Lap_{a^b}+1)^{-1}-[\chi_0,\Lap_{a^b}](-\Lap_{a^b}+1)^{-2}.
	\end{align*}
	The operator $B$ is bounded on $\cL^2$. 
	Since $\chi_0$ has its support bounded $F_\Si$, the operator $K$ is a compact operator. It follows that $\Lap_{a^b}\chi_0(-\Lap_{a^b}+1)^{-2}=BK$ is compact. This proves the claim.

3. We will use a modified version of  Weyl's Theorem for relative compact perturbations:
\begin{proposition}[\cite{MR0493421} Chapt. XIII.4, Cor. 3]\label{mWeyl}
	Let $H$ and $W$ be two self-adjoint operators s.th. $W(H+1)^{-n}$ is compact for some $n\ge1$. Then $$\si_{\rm ess}(H+W)=\si_{\rm ess}(H).$$
\end{proposition}
We use this proposition with $W=R$ from \eqref{R} and $H=-\Lap_{a^b}$. Then it follows from the claim proved in Step 2 and \propref{mWeyl} that 
\begin{align}\label{ess-spec1}
	\s_{\rm ess}(-\Delta_{a^b})=\s_{\rm ess}\del{\sum_{i=1}^mT_i}.
\end{align}

4. Next, for $i=1,\ldots, m$, we denote by $-\Delta_{a^b}\big\vert_{U_i'}$ the restriction on $U_i'$ with Dirichlet boundary condition on $\di U_i'$.

Claim: 
\begin{equation}\label{3.102}
	\s_{\rm ess}\del{\sum_{i=1}^m T_i}\subset  \bigcup_{i=1}^m \s_{\rm ess}(-\Delta_{a^b}\big|_{U_i}).
\end{equation}

Proof of the claim: By construction, $U_i'\cap U_j'=\emptyset$ for $1\le i\ne j \le m$. Hence, one can show using Weyl's criterion that 
\begin{equation}\label{3.100}
	\si_{\rm ess}\del{\sum_{i=1}^m T_i}=\bigcup_{i=1}^m \si_{\rm ess}\del{T_i}.
\end{equation}

Next, we show that for  each $i=1,\ldots,m$, 
\begin{equation}\label{3.101}
	\si_{\rm ess}\del{T_i}= \si_{\rm ess}\del{-\Delta_{a^b}\big|_{U_i}}.
\end{equation}
Indeed, for each $i=1,\ldots,m$, we write \begin{equation}\label{Ti'}
	T_i =-\Delta_{a^b}\big|_{U_i}-\Delta_{a^b}\big|_{U_i'\setminus U_i} +K_i\quad  \text{ defined on the domain } D(\Delta_{a^b}\big|_{U_i'}),
\end{equation} where $K_i$'s are defined by this relation. Explicitly, 
\begin{equation}\label{Ki}
	K_i=(-1+\chi_i)\Lap_{a^b}\chi_i-\Lap_{a^b}(1-\chi_i)\quad\text{acting on }U_i'.
\end{equation}

 Take $\lam \in \s_{\rm ess}(T_i)$, and let $\{u_n\}$ be a Weyl sequence for $T_i$ and $\lam$, i.e. 
 \begin{equation}\label{weyl}
 	\|u_n\|_{L^2}=1,\quad u_n \ra 0\ \text{ weakly in $L^2$, }\, \|(T_i - \lam) u_n\|_{L^2} \ra 0.
 \end{equation} Since $\supp u_n\subset  U_i'$, each element in the sequence $\{u_n\}$  satisfies the Dirichlet boundary condition on $U_i'$. 

	By the choice of $\chi_i$, see \eqref{chi2}, the factor $1-\chi_i$ vanishes away from the bounded set $U_i'\setminus U_i$. Hence, it follows from \eqref{Ki}--\eqref{weyl} that $$K_i u_n\ra 0 \quad \text{strongly in $L^2$ as $n\to\infty$ for every }i=1\ldots,m.$$ 
	Similarly, we have 
$$-\Delta_{a^b}\big|_{U_i'\setminus U_i} u_n\ra 0 \quad \text{strongly in $L^2$ as $n\to\infty$ for every }i=1\ldots,m.$$
	Hence, we conclude from expansion \eqref{Ti'} and assumption \eqref{weyl} that  \begin{equation}\label{this}
\begin{aligned}
			&(-\Delta_{a^b}\big|_{U_i}-\lam) u_n\\=&(T_i-\lam) u_n - (K_i-\Delta_{a^b}\big|_{U_i'\setminus U_i}) u_n\ra 0\text{  strongly for each $i=1,\ldots,m$}.
\end{aligned}
	\end{equation}By Weyl's criterion, this implies $\lam \in \s_{\rm ess}(-\Delta_{a^b}\big|_{U_i})$. Thus \eqref{3.101} holds.

Eqs. \eqref{3.100}--\eqref{3.101}  imply  \eqref{3.102}. This proves the claim.
	
	5.
	Running the argument from Step 4 backwards gives the other inclusion, i.e. 
\begin{equation}\label{ess-spec2}
			\si_{\rm ess} \del{\sum T_i}= \si_{\rm ess}\del{ \sum -\Delta_{a^b}\big|_{U_i}}.
\end{equation}
	 Relations \eqref{ess-spec1} and \eqref{ess-spec2} together imply \eqref{ess-spec-Lapl1}. 
\end{proof}

\begin{lemma}\label{lem:ess-spec-Lapl2}For every $i=1,\ldots,m$, we have $$\s_{\rm ess}(-\Delta_{a^b}\big|_{U_i})\subset[1/4+b^2, \infty).$$ \end{lemma} 
\begin{proof} 
By construction, the map $\varphi_i$ in \eqref{varphi} maps
$U_i$ isometrically onto the half-cylinder $Z_i$ in \eqref{Xc}.
This transformation maps $-\Delta_{a_b}\big|_{U_i}$ unitarily to another operator, say $h_i$, acting on $L^2(Z_i)$ with the Dirichlet boundary conditions on $\p Z_i=\{z\in \C: \Im z=s_i\}/\Z$. Thus, we have
\begin{align}\label{ess-spec-Lapl3} \s_{\rm ess}(-\Delta_{a^b}\big|_{U_i})\subset \s_{\rm ess}(h_i).\end{align}
It remains to compute the spectra of $h_i$.

First, we note that
by the invariance of the magnetic Laplacian, see \eqref{3.8}, 
the spectral properties of $-\Lap_{a^b}$ are not affected 
under any isometric transform $\g\in SL(2,\Rb)$
acting on $\Hb$. Hence, 
up to an initial gauge transform, we can assume
$h_i$ is of the form \eqref{3.7}, 
now acting on the half-cylinder $Z_i$.

Next, we pass from $h_i$ to another unitary equivalent operator 
$$p:=y^{-1}h_i y =-y (\p_x^2 + \p_y^2)y +2i b y\p_x+b^2$$  acting on $L^2(Z_i, d x d y)$, with the Dirichlet boundary conditions on $\p Z_i$.
Then
\begin{align}\label{ess-spec-Lapl4}  \s_{\rm ess}(h_i)=\s_{\rm ess}(p).\end{align}
Now, we apply the Fourier transform in $x$ for the operator $p$ to obtain the decomposition 
$$p =\bigoplus_{k\in \Z} p_k^b,$$
where $p_k^b$ is the  operator \begin{equation}\label{pkb}
	p_k^b : =y (- \p_y^2 + k^2)y  - 2 b y k+b^2
\end{equation}
 acting on $L^2(\R_s)$, with $\R_s:=\{y\in \R: y>s\}$. 

For $k\ne0$, \eqref{pkb} can be written as $$p_k^b : =-y \p_y^2 y + k^2y^2  - 2 b y k+b^2=-y \p_y^2 y + k^2 (y  -  b/k)^2,$$ and therefore we have the estimate
\begin{equation}\label{3.14}
	p_k^b\ge p_0^b\quad (k\ne0).
\end{equation}
By \eqref{3.14}, we have $$\si(p_k^b)\subset [\si^b,\infty),\quad \si^b:=\inf\si(p_0^b\vert_{\Rb_s}),$$
and therefore
 \begin{equation}\label{rel-p}
	\s (p)=\bigcup_{k\in \Z} \s (p_k^b\big|_{\R_s})\subset [\si^b,\infty).
\end{equation}  
Furthermore, by Hardy's inequality
$$- \p_y^2\big|_{\R_s} - \frac14\frac{1}{y^2}\ge 0,$$
we have the lower bound
$$p_0^b : =-y \p_y^2 y+b^2\ge \frac14+b^2.$$
This shows that $\si^b\ge 1/4+b^2$, and therefore by \eqref{rel-p}, it follows that 
$\s(p)\subset[1/4+b^2, \infty)$.
This, together with relations \eqref{ess-spec-Lapl3}--\eqref{ess-spec-Lapl4}, proves the lemma. \end{proof}

\begin{proof}[Proof of  Theorem \ref{thm3.1}, Part (d)] 
Lemmas \ref{lem:ess-spec-Lapl1} and \ref{lem:ess-spec-Lapl2} together imply 
\begin{equation}\label{inclusion1}
	\s_{\rm ess} (-\Lap_{a^b})\subset [1/4+b^2, \infty).
\end{equation}
{Here we note that as far as the bifurcation argument in \secref{sec:4} is concerned, a lower bound for the essential spectrum such as \eqref{inclusion1} suffices for our purposes.}

To prove the  inclusion  \begin{align}\label{inclusion2}
	\s_{\rm ess} (-\Lap_{a^b})\supset [1/4+b^2, \infty),
\end{align} we compute the generalized eigenfunctions of  $-\Delta_{a^b}$ in the space of square-integrable equivariant functions. {The details are delegated to the end of Appendix~\ref{sec:2.4}.}

%
%
%
%

Combining \eqref{inclusion1}--\eqref{inclusion2} proves 
Part (d) of \thmref{thm3.1}.
\end{proof}

{\begin{proof}[Proof of  Proposition \ref{prop:loc-min}]
Consider the Ginzburg-Landau energy \eqref{1.1}.
The Hessian of $\cE(\cdot, h_r)$ at $(0,a^{b_r})$ is given by
		\begin{equation}
			\label{dG}
	L:=		\diag(-\Lap_{a^{b_r}}-\kappa^2 , d^*d):X^s\to X^{s-2},
		\end{equation}
	which can be computed as in \eqref{4.3}.
		By Theorem \ref{thm3.1}(b) and rescaling, $-\Lap_{a^{b_r}}\ge b_r$. 
Thus,
		\begin{align}\label{balge}
				-\Lap_{a^{b_r}}-\kappa^2\ge0\iff b_r-\kappa^2\ge0 \iff b_r \ge \kappa^2.
		\end{align}
		Next, by Proposition \ref{prop3.1}, $d^*d\ge 0$ on $\vec\cH^s$, $ s\ge2$. This, together with \eqref{balge} and formula \eqref{dG}, implies
		\begin{align}
			\label{dG-posit}
			L\ge 0\ \text{ if and only if }\ b_r\ge \kappa^2.
		\end{align}
		This implies Proposition \ref{prop:loc-min}. \end{proof}}

\section{Bifurcation Analysis} \label{sec:4}
In the previous sections, we have found that, for the constant curvature connection $a^b$
on the unitary line bundle $E$ over a Riemann surface $\Si$,
the ground state energy of the magnetic Laplacian
$-\Lap_{a^b}$ equals to $b$.
Moreover, 
the parameter $b$ is determined by the degree of $E$ and the signature of $\Si$.

{
Let $n=\deg E$. Recall that $\abs{\Si}$ is the area of $\Si$ w.r.t.~the standard hyperbolic metric, and $b=2\pi n /\abs{\Si}$ is the critical value of the average field strength. 
In this section, we construct 
solutions to the rescaled GL equation, \eqref{2.13},
for scaling parameter  $r$ close to $ b/\kappa^2$. 
This emerging solution corresponds to a second order phase transition as the applied field strength is lowered past the critical value $b=2\pi n /\abs{\Si}$.

We follow the general approach of \cite{MR4049917}. 
Here we note two additional difficulties specific to our situation: 
}
\begin{enumerate}[label=(\alph*)]
	\item If $b=1/2$, e.g. when $\deg E=1$ and $\Si=\Hb/\Ga(3)$ (in which case $\abs{\Si}=4\pi$ by the Gauss-Bonnet formula \eqref{2.8N}),
	then ground state energy  $b$ is embedded in the essential spectrum of $-\Lap_{a^b}$.

	\item In general, by \thmref{thm3.1}, Part (c), the lowest eigenvalue of $-\Lap_{a^b}$
	is not simple. 
\end{enumerate}

{These lead to subtle technical issues due to bifurcations with higher multiplicity and bifurcation from essential spectrum. We sidestep those issues here by assuming that $\dim(-\Lap_{a^b}-b)=1$ and $b\ne 1/2$ (see \eqref{b0cond}),}
{and explain how to overcome the second one in \secref{secExt}.}

{An explicit example satisfying \eqref{Kcond} is given in \eqref{1.6'}.}
{To see this, we note that by \thmref{thm3.1}, condition
	\eqref{Kcond} holds if $b\ne 1/2$ and the complex 
	vector space $\Sc_k(\Si)$ of cusp forms on $\Si$ with
	weight $k=2b=4\pi\deg E/\abs{\Si}$
	is one-dimensional. Let $\Si=\Hb/\Ga(N),\,N\ge2$ be
	a non-compact Riemann surface,
	where $\Ga(N)$ is the principal
	congruence subgroup of level $N$ (see \secref{sec:2.1} for definitions). 
	Using classical dimension formulae
	found in \cite{MR0314766}*{Secs. 1.6, 2.6},
	we find that 
	if $k$ is even, i.e. $b\in\Z$, then 
	\begin{equation}\label{dimSk}
		\dim (\Sc_k(\Si))
		=\left\{\begin{aligned}
			&(k-1)(g_N-1)+\frac{km_N}{2}\quad (k=4,6,\ldots),\\
			&g_N\quad (k=2),
		\end{aligned}\right.
	\end{equation}
	where $g_N$ (resp. $m_N$) is the number of genera (resp. distinct cusps) of $\Si$, 
	given by \eqref{2.5}--\eqref{2.6}. We seek integer solution $(k,N)$ to the equation
	\begin{equation}\label{D.7}
		\dim (\Sc_k(\Si))=1.
	\end{equation}
For $N\ge3,\,k=2$,  \eqref{D.7} reduces to 
	\begin{equation}
		\label{D.8}
		\frac{N-6}{12}m_N=0.
	\end{equation}
	The only solution to \eqref{D.8} is $N=6$ and, for $k=2$, this gives $\deg E=12$.}
	

\subsection{Setup}
Recall $X^s=
\H^s\times \vec\H^s$ is the space of (section, connection)-pairs of order $s$,
defined in \eqref{X},
and by convention $X^0=\Lc^2\times \vec \Lc^2$.

Define a nonlinear map $F$ as
\begin{equation}
	\label{4.1}
	\fullfunction{F}{X^s\times \Rb}{X^{s-2}}{(\psi,\alpha,r)}{\del{-\Lap_{a^b+\alpha}\psi +\kappa^2(\abs{\psi}^2-r)\psi,
			d^*d \alpha-PJ(\psi,\alpha)}},
\end{equation}
where
\begin{align}
	&J(\psi,\alpha):=\Im(\bar\psi\grad_{a^b+\alpha}\psi) \text{ is the r.h.s. 
of the second equation in \eqref{GL},}\label{J}\\
&P: \vec \cH^s\to \vec \cH^s\text{ is the projection onto the space of co-closed $1$-forms}.\label{P}
\end{align}
The map $F$ is the central object in the remaining sections.

\begin{remark}\label{rem4.1}
By   definition \eqref{4.1} and direct computation, one finds   that the rescaled equation
	\eqref{2.13} has solution $\del{\psi, a^b+\alpha,r}$ if and only if 
	\begin{equation}
		\label{4.2}
		F(\psi,\al,r)=0,
	\end{equation}
	and $J(\psi,\al)$ is co-closed. The latter holds
	so long as the {first GL equation in \eqref{2.13}} holds. 
	A proof of this fact is found in \cite[{Prop.5.1}]{MR4049917}. {Hence 
	it suffices to consider \eqref{4.2} only. }
\end{remark}

Next, by \eqref{3.1}, eq.~\eqref{4.2} has the trivial solution $(\psi,\al,r)=(0,0,b/\kappa^2)$:
\begin{align}\label{4.4s}
	F(0,0,b/\kappa^2)=0,
\end{align}
Below, we seek non-trivial solutions to \eqref{4.2}, i.e.~solutions to \eqref{4.2} with variable curvature, in a small neighbourhood around the trivial one. We write \begin{align}\label{4.5s}
	u:=(\psi,\alpha)\in X^s.
\end{align}
{Recall also that $K=\Null(-\Lap_{a^b}-b)$, c.f. \eqref{3.9},} and the critical threshold $\kappa_c (r)$ is defined in \eqref{kappac}.

The main result of this section is the following:

\begin{proposition}
	\label{thm4.1}
	Suppose{ $b\ne1/2$, $\dim K=1$,}  
	{and $r>0$ satisfies (c.f.~\eqref{1.5})
\begin{align} 
	\label{1.5'}  
&0<\abs{\kappa^2r-b}\ll1. 
\end{align}}
%
%
Then there exists {{a family  of non-trivial solutions }}
$$(u_{s},r_s),\qquad s\in\Rb,\ 0<\abs{s}\ll1,$$ to \eqref{4.2} in a small neighbourhood
of $(0,b/\kappa^2)$ in $X^k\times \Rb_{>0}$, $k\ge2$.

{The solution $u_s$ is unique, up to a  gauge symmetry, in a small neighbhourhood $U\subset X^k$ around $0$.

	Furthermore, $r_s$ has the following expansion,
	\begin{align}
		\label{4.11s}
		r_s=b/\kappa^2 + O(\abs{s}^2),
	\end{align}
and similarly for derivatives.
}

\end{proposition}

The remainder of this section is devoted to the proof of \propref{thm4.1}.

First, we summarize the key properties of the  map
$F(u,r)$ from \eqref{4.1} in the following proposition.
Below, we identify $\cH^s$ with a real Banach space through $\psi \leftrightarrow (\Re \psi,\Im \psi)$ and view $F: X^s\times \Rb\to X^{s-2}$ as a map between real Banach spaces.
{From now till the end of this paper, we assume 	$s\ge2$ in \eqref{4.1}.}
\begin{lemma}
	\label{prop4.1}
We have
	\begin{enumerate}
		\item The map $F$ is $C^2$ from the real Banach spaces $X^s\times \Rb$ to $X^{s-2}$, $s\ge2$;
		\item {$F$ has gauge symmetry in the sense that for every $\theta\in\Rb$, we have
		\begin{align}\label{invar}
			[F(e^{i\theta}s,t,r)]_\psi =e^{i\theta}[F(s,t,r)]_\psi,\quad 	[F(e^{i\theta}s,t,r)]_\al=[F(s,t,r)]_\al;
		\end{align}}
		\item $F(u,r)=0$ has the trivial branch of solution $(0,r),\,r>0$.
	\end{enumerate}
\end{lemma}
\begin{proof}
{	The first claim follows from the fact that the map $F$ is a polynomial in $\psi$, $\bar\psi$, $\al$ and their derivatives up to the second order, together with the Sobolev inequalities and properties of fractional derivatives. }
	{The second claim follows directly from definition \eqref{4.1}.
	The last claim is a rephrasing of \eqref{4.4s}.}
\end{proof}

Next, consider the linearized operator of $F$ at the trivial branch 
$(u,r)=(0,r)$, given explicitly by (c.f. \eqref{3.2}--\eqref{3.3})
\begin{equation}
	\label{4.3}
	d_uF(0,r)=\diag(-\Lap_{a^b}-\kappa^2r,d^*d):X^s\to X^{s-2}.
\end{equation}
{For fixed $r$, this map is well-defined as the partial Fr\'echet derivative of $F$ at $u=0$. Moreover,  $d_uF(0,r)$ is continuous from $X^s\to X^{s-2}$ for $s\ge2$. }

The operator \eqref{4.3} enters the l.h.s.~of the linearized equations \eqref{3.2}--\eqref{3.3}.
Put 
\begin{equation}\label{4.4}
	N_r:=\Null d_uF(0,r).
\end{equation}
By Proposition \ref{prop3.1} and \corref{cor:Delta_A-part_A}, we have
\begin{equation}
	\label{4.4'}
	N\equiv N_{b/\kappa^2}=K\times \Omega,
\end{equation}
where $\Om$ is given in \eqref{3.5} and $K$ in \eqref{3.9}.

The goal now is to show that a non-trivial branch of solution to \eqref{4.2}
bifurcate from the space $N$ if $0<\abs{\kappa^2r-b}\ll1$.

\subsection{Lyapunov-Schmidt Reduction}\label{sec:4.2}
In this subsection, we  use the Lyapunov-Schmidt reduction to reduce the infinite-dimensional problem \eqref{4.2}
to a finite-dimensional one.

\DETAILS{At this point the geometry of the surface $\Si$ enters the picture.
Indeed, by \thmref{thm3.1}, Parts (b)-(d), the lowest eigenvalue of $-\Lap_{a^b}$
is embedded in the essential spectrum if and only if $b=1/2$.
Hence we must treat two cases separately depending on the value of $b$.}

Define a linear operator $Q: X^s\to X^s$ by
\begin{equation}
	\label{4.5}
	Q:=\left.\begin{aligned}
		&\frac{1}{2\pi i}\oint R(z)\,dz\oplus Q'.
	\end{aligned}\right.
\end{equation}
Here, 
the operator $Q':\vec\H^s\to \vec\H^s$ is the orthogonal projection onto $\Om$ defined in \eqref{3.5}, which can be identified with 
the space  of equivariant harmonic $1$-forms. 
$R(z)$ is the resolvant of $-\Lap_{a^b}-b$ at $0$,
which is well-defined    if $b\ne1/2$.

By construction, $Q$ given in \eqref{4.5}
is an orthogonal projection onto the space $N\equiv N_{b/\kappa^2}$ from \eqref{4.4'}.

Define $v=Qu,\,w=Q^\perp u$, where $Q^\perp=1-Q$ is the projection onto $N^\perp
\subset X^s$. {Then the key equation \eqref{4.2} is equivalent to the following two equations,}
\begin{align}
	QF(v+w,r)&=0, \label{4.6}\\
	Q^\perp F(v+w,r)&=0\label{4.7}. 
\end{align}

\begin{lemma}
	\label{prop4.2}{Suppose $b\ne1/2$.}	For every $(v,r)\in N\times \Rb_{>0}$ with 
$$\norm{v}_{X^s}+\abs{\kappa^2r-b}\ll1,$$
eq.~\eqref{4.7} has a unique solution $w=w(v,r)\in \ran Q^\perp\subset X^s$ which satisfies  \begin{equation}\label{4.12a}
	w=O(\|v\|_{X^s}^2),
\end{equation}
{and similarly for the derivatives of $w$.}

\end{lemma}
\begin{proof}
1. We first prove the existence and uniqueness of solution to \eqref{4.7}.	 By \lemref{prop4.1} Parts {(1), (3)}, and the Implicit Function Theorem,
	 {eq.~\eqref{4.7} has a unique solution} $w=w(v,r)\in N$ in a small
	neighbourhood of $(v,r)=(0,b/\kappa^2)$ provided
	the partial Fr\'echet derivative
	$d_wQ^\perp F(0,b/\kappa^2):N^\perp\subset X^s\to X^{s-2}$ is invertible.

	The partial Fr\'echet derivative $d_wQ^\perp F$ evaluated at $(0,b/\kappa^2)$ is given by
	the diagonal operator 
	\begin{equation}
		\label{4.8}
		d_wQ^\perp F(0,b/\kappa^2)=Q^\perp\diag(-\Lap_{a^b}-b,d^*d).
	\end{equation}

	We first consider the $\psi$-component of this operator.
	
 {By assumption,} the ground state energy $b\ne1/2$.
	Then by \thmref{thm3.1}, Parts (c)-(d), 
	 $b$ is an isolated eigenvalue
		of $-\Lap_{a^b}$ (see \remref{remEmb}).
		In this case, by elementary spectral theory (e.g. \cite[{Thm. 6.7}]{MR1361167}), the operator $-\Lap_{a^b}-b$
		is  invertible on $K^\perp$. 

	 It remains to consider  the $\al$-component of the diagonal operator \eqref{4.8},
	namely $d^*d$.
	On the space of co-closed
	$1$-forms,
	$d^*d$ equals to the Hodge Laplacian. {By the lower bound on the essential spectrum of the Hodge Laplacian acting on $1$-forms, proved in e.g.~\cite[Prop.~5.1]{MR729759}), we have $\inf\si_{\rm{ess}}(d^*d\vert_{\Om^\perp})\ge \frac14$. Thus $d^*d\vert_{\Om^\perp}$ is invertible. }

	It follows that $d_wQ^\perp F(0,b/\kappa^2)$ is invertible on $N^\perp$ as desired.
	
2.		To prove estimate \eqref{4.12a}, consider the equation \eqref{4.7} satisfied by $w$,
	which we rewrite as
	\begin{equation}
		\label{5.7}
		Q^\perp F(v,r)+\b L_{v,r}w +N_{v,r}(w)=0,
	\end{equation}
	where
	$\b L_{v,r}:= Q^\perp d_uF(v,r) Q^\perp$, and $N_{v,r}(w)$ is the nonlinearity defined through
	this relation. 
	
	{The explicit expansion for $\bar L_{0,r}$ follows from \eqref{4.3}. By \propref{prop3.1} and \thmref{thm3.1}, the operator $\b L_{0,b/\kappa^2}\ge\delta Q^\perp$ for some $\delta>0$ and  is therefore invertible. By this fact and elementary
	perturbation theory, we find that 
	$\b L_{v,r}$ is invertible for
	\begin{align}
		\label{vInvCond}
		\norm{v}_{X^s}+\abs{\kappa^2r-b}\ll1,
	\end{align}}
	with 
	\begin{equation}\label{5.8}
		\norm{\b L_{v,r}^{-1}}_{X^{s-2}\to X^s}\ls 1  .
	\end{equation}
	Direct computation using the definition \eqref{4.1} of $F$ and the fact $v\in N$ (see \eqref{4.4}--\eqref{4.4'}) shows that 
	\begin{align}
		\label{5.9}
		\norm{[ F(v,r)]_\psi}_{\cH^{s-2}}&\ls \norm{[v]_\psi}_{\cH^s}(\norm{[v]_\al}_{\vec \cH^s}+\norm{[v]_\psi}_{\cH^s}^2),\\
		\label{5.10}
		\norm{[ F(v,r)]_\al}_{\vec \cH^{s-2}}&\ls \norm{[v]_\psi}_{\cH^s}^2( 1+\norm{[v]_\al}_{\vec \cH^s}),\\
		\label{5.11}
		\norm{N_{v,r}(w)}_{X^{s-2}}&\ls \norm{w}_{X^s}^2.
	\end{align}
	Here and below,
	for a  vector $u\in X^s$ we write $u=([u]_\psi,[u]_\al)$.

	Now we rewrite \eqref{5.7} as the fixed point equation 
	\begin{equation}
		\label{5.12}
		w=-\b L_{v,r}^{-1}( Q^\perp F(v,r)+ N_{v,r}(w)).
	\end{equation}
Applying  \eqref{5.9}--\eqref{5.11} to \eqref{5.12} and using triangle inequality,
		we find 
\begin{align}\label{6.15a}
	&\norm{[w]_\psi}_{\cH^s}\ls \norm{[v]_\psi}_{\cH^s}(\norm{[v]_\al}_{\vec \cH^s}+\norm{[v]_\psi}_{\cH^s}^2),\\
	\label{6.15b}
	&\norm{[w]_\al}_{\vec \cH^s}\ls  \norm{[v]_\psi}_{\cH^s}^2( 1+\norm{[v]_\al}_{\vec \cH^s}),
\end{align}
		provided $\norm{v}_{X^s}+\abs{\kappa^2r-b}\ll1$. 
	Estimates \eqref{6.15a}--\eqref{6.15b} give \eqref{4.12a}.
	{To obtain estimates on the derivatives of $w$, we differentiate \eqref{5.7} and then proceed with the resulting equation as above using estimates on $N_{v,r}(w)$ and its derivatives.}
\end{proof}

\subsection{The Bifurcation Equation}\label{sec:4.3}
In this section we prove \propref{thm4.1}, by solving the bifurcation equation \eqref{4.6}. 


\begin{proof}[Proof of \propref{thm4.1}]
	1. Using \lemref{prop4.1}, we can now plug the solution $w=w(v,r)$ to \eqref{4.7}
	back into \eqref{4.6}, and get the bifucation equation
	\begin{equation}
		\label{4.19}
		QF(v+w(v,r),r)=0.
	\end{equation}
Recall here $Q$ is the orthogonal projection onto $N:=\Null  d_uF(0,b/\kappa^2)$
defined in \eqref{4.5}.
	Solving \eqref{4.19} in $v$ and $r$ amounts to
	solving \eqref{4.6}--\eqref{4.7}, and therefore 
	gives a solution $u=v+w(v,r)$ to
	\eqref{4.2}.
	
	2. Let
	\begin{equation}\label{4.20}
{		 \xi  \quad \text{and} \quad\eta_k,\,1\le k\le \dim \Om}
	\end{equation}
	be some orthogonal bases of $K$ and $\Om$,  defined respectively in \eqref{3.9} and \eqref{3.5},  {with $\br{\abs{\xi}^2}\equiv \frac{1}{\abs{\Si}}\int \abs{\xi}^2 =1$. }
	
{Here $K\subset \H^s$ is a  complex vector subspace with dimension $1$ by assumption,  while $\Om\subset \vec \H^s$ is  real vector space with finite dimensions \cite{MR729759}.}

	Let $s$ and $t=(t_1,\ldots,t_{\dim_\Om})$
	{respectively be {the {complex and} real coefficients} of vectors in $K$ and $\Om$  w.r.t.~the bases \eqref{4.20}.
For $v=(\phi, \gamma)\in K\times \Omega$, we will use the parametrization 
\begin{equation}\label{4.21}
		\phi=\phi(s)\equiv s\xi\in K,\quad \gamma=\gamma(t)\equiv \sum_{k=1}^{\dim \Om} t_k\eta_k\in \Om.
\end{equation}
}
	3. 
	Next, let $V$ be a sufficiently small neighbourhood of $(s,t,r)=(0,0,b/\kappa^2)$ in $\Rb\times \Rb^{\dim \Om}\times \Rb_{>0}$ {(note that hereafter, $s$ is real-valued in $V$)}. For every $(s,t,r)\in V$, we consider \begin{align}\label{4.26s}
		u_{str}:=	(\psi_{str},\al_{str}) = v_{st}+w_{str},
		\end{align}
	 where $v_{st}=(\phi(s),\g(t))$ is parametrized as \eqref{4.21} and $w_{str}=w(v_{st},r)$ is the solution found in \lemref{prop4.2}, satisfying, by estimates  \eqref{6.15a}--\eqref{6.15b},
	\begin{align}
		\label{4.26ss}
		w_{str}=(O(\abs{s}^3)+O(\abs{s}\abs{t}),O(\abs{s}^2)),
	\end{align}
{and similarly for its derivatives in $s,t,r$ on $V$.}

By \lemref{prop4.1}, Part (1),  
we can expand $$F(s,t,r)\equiv F
(u_{str},r)$$ 
{as a $C^2$ map from $V\subset\Rb\times \Rb^{\dim \Om}\times \Rb_{>0}$ to the real Banach space $X^{s-2}$, $s\ge2$.}


	By the gauge symmetry \eqref{invar} of $F$, it is not hard to check that any solution to the equation \begin{align}\label{4.7'}
		F(s,t,r)=0
	\end{align} in $V$ gives rise to a circle of solutions to the equation $F(u,r)=0$, with $u=v_{e^{i\theta}s,t}+w_{e^{i\theta}s,t,r}$ and any $\theta\in\Rb$. Thus, our goal now is to solve \eqref{4.7'} in $V$.

To begin with, we derive an explicit expression of the map $F(s,t,r)$. 
Estimate \eqref{4.26ss}, together with definition \eqref{4.26s}, yields
	\begin{align}\label{psiform}
		\psi_{str}=&\phi(s)+O(\abs{s}^3)+O(\abs{s}\abs{t}),\\\al_{str}=&\g(t)+O(\abs{s}^2),\label{alForm}
	\end{align}
{with corresponding estimates for derivatives in $s,t,r$ on $V$. }
Plugging \eqref{psiform}--\eqref{alForm} into \eqref{4.1}, observing that the nonlinear term in the first component of \eqref{4.1} is cubic in $\psi$, and using that \begin{align}
		(\Lap_{a^b+{\alpha_{str}}}-\Lap_{a^b})\xi =O(\abs{t})+O(\abs{s}^2),
	\end{align}
we  obtain
 	\begin{align}
 	\label{4.22}
 	[F(s,t,r)]_\psi&=s(-\Lap_{a^b}-\kappa^2r)\xi+O(\abs{s}\abs{t}){+O(\abs{s}^3),}
 	\end{align}
{with suitable estimates on the error terms.}
 
Next, since $J(\psi,\al)=\Im(\bar\psi\grad_{a^b+\alpha}\psi)$ (see \eqref{J}) is quadratic in $\psi$, we find, with  $\psi_{str}$ as in  \eqref{psiform}, \begin{align}\label{Jexp}
	J(\psi_{str},\al)=\abs{s}^2J(\xi,0)-\abs{s}^2\abs{\xi}^2\al +O(\abs{s}^2\abs{t}^2)+O(\abs{s}^4\abs{t})+O(\abs{s}^6).
\end{align} For $\al=\al_{str}$ as in \eqref{alForm}, expansion \eqref{Jexp} becomes
\begin{align}\label{Jexp'}
	J(\psi_{str},\al_{str})=\abs{s}^2J(\xi,0)-\sum_kt_k\abs{s}^2\abs{\xi}^2\eta_k+O(\abs{s}^2\abs{t}^2)+O(\abs{s}^4).
\end{align} 
Plugging \eqref{Jexp'} into \eqref{4.1} yields
 	\begin{align}
		\label{4.23}
		[F(s,t,r)]_\al=d^*d\al_{str}- \abs{s}^2J(\xi,0)-\sum_k t_k\abs{s}^2\abs{\xi}^2\eta_k +O(\abs{s}^2\abs{t}^2)+O(\abs{s}^4),
	\end{align}
{with corresponding estimates for derivatives in $s,t,r$ on $V$.} 
Moreover, in addition to using \eqref{psiform}--\eqref{alForm}, we can also eliminate
the even order terms in expansion \eqref{4.22} and odd order terms in 
\eqref{4.23}, counting $\abs{s}$ and $\abs{t}$ as of the orders $1$ and $2$, respectively, by the gauge invariance \eqref{invar}.

4. Next, 	let $G(s,t,r)=QF(u_{str},r)$, where $u_{str}$ is defined in \eqref{4.26s} and write $G=([G]_\psi,[G]_\al)$. 
Our goal now is to derive a more explicit formula for $G$ in terms of the 
	parametrization \eqref{4.21}. 
	
	{Firstly, by \lemref{prop4.1} part (1) and the chain rule for Fr\'echet derivatives,
		$G(s,t,r)$ is a $C^2$ map from $V\subset\Rb\times \Rb^{\dim \Om}\times \Rb_{>0}$ to the finite-dimensional real vector space $N\subset X^{s-2}$, $s\ge2$.} Secondly, by definition, 
	we have $-\Lap_{a^b}-b\vert_{K}=0$.
	Lastly, so long as $\xi$ solves \eqref{3.2}, i.e. $\xi\in K$,
 by \cite[{Props. 5.1, 5.3}]{MR4049917}, the supercurrent
	 $J(\xi,0)$ is both co-closed and co-exact. 
	 By the characterization of the space $\Om$ in \propref{prop3.1}, this implies 
 $Q'd^*d\al_{str}=Q'J(\xi,0)=0$.
 
{Hence, using these facts, we can rewrite \eqref{4.22} and \eqref{4.23} as}
	\begin{align}
		\label{4.24}
		{[G(s,t,r)]_\psi}&=s\del{\kappa^2r-b} \xi
		{+O(\abs{s}\abs{t}){+O(\abs{s}^3)}},\\
		\label{4.25}
		[G(s,t,r)]_\al&=\abs{s}^2{\sum_kt_kQ'\del{\abs{\xi}^2\eta_k}+O(\abs{s}^2\abs{t}^2)+O(\abs{s}^4)}, 
	\end{align}
{with corresponding estimates for derivatives in $s,t,r$. }

5. {Now we investigate the equations 
	\begin{align}
	{  \frac1s[G(s,t,r)]_\psi=}&0,\label{443}\\ \frac{1}{\abs{s}^2}[G(s,t,r)]_\al=&0.\label{444}
	\end{align}
Clearly, any solution to \eqref{443}--\eqref{444} is also a solution to the bifurcation equation \eqref{4.19}. 
In view of the parametrization \eqref{4.20}, taking inner product of \eqref{443}--\eqref{444} w.r.t. $\xi $ and $\eta_k$, we can write these equations as
	\begin{align}
		\label{4.26}
		 { \kappa^2r-b }{+R_\psi(s,t,r)}&=0,\\
		\label{4.27}
		\sum_{l=1}^{\dim \Om}B_{kl}t_l{+R_{\al,k}(s,t,r)}&=0,\quad k=1,\ldots,\dim \Om.
	\end{align}
{Here, the remainders $R_\psi$ and $R_{\al,k}$ are $C^2$ functions from $V$ to $\Cb$ and $\Rb^{\dim\Om}$, respectively, since $G$ is a $C^2$ map from $V$ to $N\subset X^{s-2},\,s\ge2$.} 
The remainders satisfy the estimates
	\begin{align}
		R_\psi(s,t,r)=& {O (\abs{t})+O(\abs{s}^2)},\label{RpsiEst}\\
		R_{\al,k}(s,t,r)=&{O(\abs{t}^2)+O(\abs{s}^2)},\label{RalEst}
	\end{align}
{and similarly for their derivatives,} and 

\begin{equation}\label{B1}
	{	B_{kl} :=\inn{\eta_k}{\xi\eta_l}_{\vec L^2}.}
	\end{equation}
}
To obtain \eqref{4.27}, we use self-adjointness of the orthogonal projection $Q':\vec\cH^s\to \vec\cH^s$, and the fact that $\eta_k\in \Om=\ran Q'$ for every $k$.


To summarize, we are now left with solving \eqref{4.26}--\eqref{4.27} in $V\subset\Rb\times\Rb^{\dim\Om}\times\Rb_{>0}$, which amounts to solving the bifurcation equation \eqref{4.19}.

Now, we view $s$ as a parameter and solve the algebraic system \eqref{4.26}--\eqref{4.27} for $(t,r)\in\Rb^{\dim\Om}\times \Rb_{>0}$.
{\begin{lemma}
		\label{thmD.1}

		{For every $s\in\Rb$ with $\abs{s}\ll1$, there exists a unique $C^2$ solution $(t,r)$ to \eqref{4.26}--\eqref{4.27} in a small neighbourhood of $(0,b/\kappa^2)\in \Rb^{\dim\Om}\times \Rb_{>0}$.
			
			Moreover, this solution satisfies}
		\begin{align}
			\label{5.1s}
			t=O(\abs{s}^2),\quad r=\frac{b}{\kappa^2}+O(\abs{s}^2), \quad \text{ as }\abs{s}\to 0,
		\end{align}
		{and similarly for their derivatives.}
	\end{lemma}
}
This proposition is proved in \secref{sec:D}.

6. \lemref{thmD.1},
together with \lemref{prop4.2}, implies that there exists a family  
\begin{align}
	\label{usSol}
	(u_s,r_s),\qquad s\in\Rb,\quad \abs{s}\ll1 ,
\end{align}
with $r_s$ satisfying \eqref{5.1s}, that uniquely solves \eqref{4.2} in a small neighbourhood of the trivial solution $(u,r)=(0,b/\kappa^2)$.

This	 completes the proof of \propref{thm4.1}.
\end{proof}

\section{{Solvability of the Bifurcation Equation} \label{sec:D}}


In this section, we solve the bifurcation equations \eqref{4.26}--\eqref{4.27} by proving \lemref{thmD.1}.
\begin{proof}[Proof of \lemref{thmD.1}]
1. 
	We first solve \eqref{4.27}.
	This equation has the trivial solution $(s,t,r)=(0,0,b/\kappa^2)$ (see \eqref{4.4s}), and the l.h.s.~of \eqref{4.27} is $C^2$ in a small neighbourhood {$V\subset\Rb\times \Rb^{\dim \Om}\times \Rb_{>0}$ around this zero.}
	
	
	We differentiate w.r.t. $t$ and use  and the remainder estimate \eqref{RalEst} to find that the Jacobian matrix of the l.h.s.~of \eqref{4.27} at this zero is given by\begin{equation}\label{B2}
		B =[B_{kl} ], 
	\end{equation}
	where $B_{kl} $ is as in \eqref{B1} and $1\le k,l\le\dim \Om.$
	
	Let ${t} = (t_1, \dots t_{\dim \Om})$ be an arbitrary non-zero real vector. 
	We compute, {for 
		any  $\om = \sum_{i=1}^{\dim \Om}t_i \eta_i  \in\Om\setminus\Set{0}$,
		\begin{align}
			\inn{Bt}{t}=& 
			\sum_{l,k=1}^{\dim \Om} 
			\inn{\eta_k}{\abs{\xi }^2\eta_l}_{\vec \cL^2}t_l t_k
			\notag\\
			=& \inn{\om}{\abs{\xi }^2\om}_{\vec \cL^2}>0.
			\label{Best}
		\end{align}
		Thus the matrix $B$ in \eqref{B2} is positive-definite  and therefore invertible.}
	
	By construction, \eqref{4.27} is a real system of $C^2$ equations posed on a real vector space, {$\Rb\times \Rb^{\dim\Om}\times \Rb$.  By the gauge symmetry \eqref{invar}, the functions $R_{\al,k}$ depend only on $\abs{s}$ for all $k$.}
	{Hence, by the invertibility of $B$ and the Implicit Function Theorem, eq.~\eqref{4.27} has a unique $C^2$ solution $t=(t_1,\ldots,t_{\dim\Om})$ in a small neighbourhood around $(s,r)=(0,b/\kappa^2)$ in $\Rb\times \Rb_{>0}$, satisfying
		\begin{equation}\label{D.2}
			t(s,r)=O(\abs{s}^2),
		\end{equation}
		and similarly for its derivatives. }
	
	We note that, by the second relation in \eqref{invar}, the solution $t$ depends only on $\abs{s}$ and $r$. 

2. Next, we define the l.h.s.~of \eqref{4.26} after plugging the solution $t$ found in Step 1 to \eqref{4.27} 
as \begin{align}\label{tGdef}
		\tilde G (s,r) =  \kappa^2r-b {+\tilde R_\psi(s,r)},
	\end{align}
	where $\tilde R_\psi(s,r):= R_\psi(s,t(s,r),r)$    satisfies,  by \eqref{RpsiEst} and \eqref{D.2},
\begin{align}
				\tilde R_\psi(s,r)=& {O(\abs{s}^2)},\label{RpsiEst'} 
\end{align}
		and similarly for its derivatives. 
	   
	   We look for solution in a small neighbourhood of $(s,r)=(0,b/\kappa^2)$  to the equation
	\begin{align}
		\label{4.26'}
\tilde G(s,r)=0.
	\end{align}
	By \eqref{4.4s} and definition \eqref{tGdef}, eq.~\eqref{4.26'} has the trivial solution $(s,r) = (0, b/\kappa^2)$.  Moreover, by the $C^2$ regularity of the solution $t=t(s,r)$ found in Step 1 above, the map $\tilde G$ is $C^2$ around this zero, and $\di_r \tilde G(0,  b/\kappa^2)=\kappa^2>0$ {by estimate \eqref{D.2} and similarly for $\di_r\tilde R_\psi$}. Thus, by the Implicit Function Theorem, we obtain a unique solution $r=r(s)$ to \eqref{4.26'} around $s=0$, satisfying
		\begin{equation}\label{rEst}
			r(s)=\frac{b}{\kappa^2}+O(\abs{s}^2),
		\end{equation}
		and similarly for its derivatives. 
This proves \lemref{thmD.1}.

\end{proof}

\section{Precise Asymptotics of the Non-trivial branch}\label{sec:5}

In this section we finish our proof of \thmref{thm1.1}.

\begin{proposition}
	\label{prop5.1}
{	Let the conditions of \propref{thm4.1} hold and let
	$$(u_s,r_s),\quad \abs{s}\ll1 $$ be the nontrivial branch of solution to \eqref{4.2}, constructed in \propref{thm4.1}.
} Then, for $u_s=(\psi_s,\al_s)$, we have 
	{\begin{align}
		\psi_s&=\phi(s)+O(\abs{s}^3),\label{5.1}\\
		\al_s&=\g(t(s))+O(\abs{s}^4).\label{5.2}
	\end{align} 
	Here $\phi$ and $\g$ are as in \eqref{4.21} and satisfy
	\begin{align}
		&(-\Lap_{a^b}-b)\phi=0,\label{5.3}\\
		&d\g=\frac{1}{2}*\del{1-\abs{\phi}^2}.\label{5.4}
	\end{align}}

Moreover, we can take $s\in\Rb_{\ge0}$ and if, in addition, $r$ satisfies (c.f.~\eqref{1.5})
\begin{align}\label{1.5''}
	(\kappa-\sqrt {b/r})(\kappa -\kappa_c ) >0
\end{align}
where $\kappa_c=\kappa_c(1)$ is given by \eqref{kappac} , then, the equation $r=r(s)$ can be solved for s to obtain $s=s(r)$ with 
\begin{align}
				\label{5.4'}
{s}^2&= \frac{\kappa^2r-b}{(\kappa^2-\tfrac{1}{2})\beta+\tfrac{1}{2}}+O\del{\del{\frac{\kappa^2r-b}{(\kappa^2-\tfrac{1}{2})\beta+\tfrac{1}{2}}}^2}, 
\end{align}
{where $\beta=\beta(1)$ is the Abrikosov function given by \eqref{beta}.}

\end{proposition}

\begin{proof}
	The solution branch $u_s$ from \eqref{usSol} is given   by 
	\begin{equation}\label{5.5}
		u_s=v_s+w(v_s,r_s),
	\end{equation}
	{where $v_s=(\phi(s),\g(t(s)))$ with $\phi$ and $\g$ given in \eqref{4.21}, $t(s)=t(s,r(s))$ and $r_s=r(s)$ given in \eqref{D.2} and \eqref{rEst}, and with $w$ satisfying \eqref{4.26ss}. This implies \eqref{5.1}--\eqref{5.2}}.
Eqs.~\eqref{5.3}--\eqref{5.4} are implicit in the proof of \propref{thm4.1}. These parts are analogous to the proofs of \cite[{Prop.5.5}, eqs.~(5.26)--(5.27)]{MR4049917}.

 It remains to prove \eqref{5.4'}.
{Observe that by gauge symmetry \eqref{invar}, 
	$\tilde R_\psi(s,r)$ in \eqref{tGdef}
	 depends on $s$ through $\abs{s}$.} 
Thus, we can assume $s\in\Rb$. {Again, by the reflection symmetry, which follows from gauge symmetry \eqref{invar}, the expansion of $r$ in $s$ contains only even powers. Thus, we have}
\begin{equation}\label{5.13}
	r=\frac{b}{\kappa^2}+R {s}^2+O({s}^4),
\end{equation}
for some  $R\in\Rb$, {and with corresponding estimates on derivatives in the remainder.}

Take $\xi\in K$ as in \eqref{4.20}. Using the expansion \eqref{5.1}--\eqref{5.2}, together with the  \eqref{5.13},
we find
\begin{equation}
\label{5.14}
\begin{aligned}
	&\inn{\xi}{(-\Lap_{a^b+\al}-\kappa^2r)\psi }\\=&\inn{\xi}{(-\Lap_{a^b}-b)\psi }+
	{s}^2(\inn{\xi}{2i\eta\cdot \grad_{a^b}\psi}-\kappa^2R \inn{\xi}{\psi})+O({s}^4)\\
	=& {s}^2(\inn{\xi}{2i\eta\cdot \grad_{a^b}\psi}-\kappa^2R \inn{\xi}{\psi})+O({s}^4).
\end{aligned}
\end{equation}
Here, $\eta:={s}^{-2}\g(t(s))$, and we move to the last line using the fact that 
$-\Lap_{a^b}$ is self-adjoint, and the equation \eqref{5.3} satisfied by $\xi$. 

Since, by construction, the pair $(\psi,a^b+\al)$ solves the rescaled Ginzburg-Landau equations \eqref{2.13},
we have
\begin{equation}\label{5.14'}
	(-\Lap_{a^b+\al}-\kappa^2r)\psi+\kappa^2\abs{\psi}^2\psi=0.
\end{equation}
Then taking the inner product
w.r.t. ${s}^{-2}\xi$ on both sides of \eqref{5.14'},  using the expansion \eqref{5.14}, and then taking $s\to0$, we find 
\begin{equation}
	\label{5.15}
	\inn{\xi}{2i\eta\cdot \grad_{a^b}\xi}-\kappa^2R \norm{\xi}_{\cL^2}^2+\kappa^2 \norm{\xi}_{\cL^4}^4=0.
\end{equation}

Next, recalling the definition $\br{f}=\frac{1}{\abs{\Si}}\int _\Si f$, we have the identity (see the equation after \cite[eq.~(5.32)]{MR4049917})
\begin{align}
	\label{5.16}
	\frac{1}{\abs{\Si}}\inn{\xi}{2i\eta\cdot \grad_{a^b}\xi}=&{\frac{1}{2}\del{ \br{\abs{\xi}^2}^2-\br{\abs{\xi}^4}} }.
\end{align}
Substituting \eqref{5.16} into \eqref{5.15} and solving for $R$, we find
\begin{align}\label{6.14s}
	{R=\frac{1}{2\kappa^2}+\del{1-\frac{1}{2\kappa^2}}\br{\abs{\xi}^4}\quad \text{since }\br{\abs
	\xi^2}=1. }
\end{align}
Now, rearranging \eqref{5.13}, we find
\begin{align}\label{sEst1}
	{s}^2 =   \frac{\kappa^2r-b}{ \kappa^2 R} +O({s}^4),
\end{align}
which, together with \eqref{6.14s} and the fact that ${\beta=\br{\abs{\xi}^4}}$ in the non-degenerate case with $\dim K=1$, yields 
\begin{align}\label{sEst2}
	{s}^2=\frac{\kappa^2r-b}{\del{\kappa^2-\frac12}{\br{\abs{\xi}^4}}{+\frac{1}{2}}}+O({s}^4).
\end{align}
With this, the equation $r=r(s^2)$ can be solved for $s^2$ to yield $s^2=s^2(r)$, with (6.6).

 Note that, {since $(\kappa^2-\tfrac{1}{2})\beta+\tfrac{1}{2}=\beta(\kappa^2-\kappa_{c}^2)$ by definitions \eqref{beta}--\eqref{kappac},  condition \eqref{1.5''} ensures} that the leading term in \eqref{sEst2} is strictly positive and so \eqref{5.4'} is self-consistent for small $s$ (c.f. \remref{rem1.1}).

This completes the proof of \propref{prop5.1}.
\end{proof}

 \begin{remark}\label{rem10}
 	{It follows from \eqref{5.4'} and condition \eqref{1.5} that $r(s)$ is locally parabolic; i.e., $r \sim C s^2$. Hence there can be at most two branches, $\pm s(r)$. However, since these would be related by a gauge transformation, up to gauge change, there will locally near $(r,s)=(b/\kappa^2, 0)$ be only one branch $s(r)$.}
 \end{remark} 

\begin{proof}[{Proof of \thmref{thm1.1}}]
	For every $s=s(r)$ with $r$ satisfying \eqref{1.5}, 
	the solution $u_s\equiv(\psi_s,\al_s)$ constructed in \propref{thm4.1} gives
	a solution to the rescaled Ginzburg-Landau equation \eqref{2.13} {on $(\Si,h_1)$ as  $(\psi_s, a^b+\al_s)$. 
			By relation \eqref{3.9}, eq.~\eqref{5.3} implies that $\xi$ is  a holomorphic section of $E$ corresponding to $a^b$. 	
			
Undoing the rescaling \eqref{scale1},}
the asymptotics \eqref{1.8}--\eqref{1.11} follow from \eqref{D.2} and \eqref{5.1}--\eqref{5.4},
	and the expansion \eqref{s} follows from \eqref{5.4'}. 
{This completes the proof of \thmref{thm1.1}.}
\end{proof}

\begin{proof}[Proof of \corref{cor1.2}]
{We first prove that, with $(\psi,a)\equiv (\psi_{s(r)},a_{s(r)})$ solving \eqref{2.13} and rescaled GL energy functional $\cE_r(\psi,a)$ from \eqref{2.12}, we have 
	\begin{equation}\label{5.17}
		\cE_r(\psi,a)=\cE_r(0,a^b)-\frac{{\abs{\Si}}}4 \frac{\abs{\kappa^2r-b}^2}{(\kappa^2-\tfrac{1}{2}){\br{\abs{\xi}^4}}+\tfrac{1}{2}}+O(\abs{\kappa^2r-b}^3),
	\end{equation}
	for $0<\abs{\kappa^2r-b}\ll1$. Here $\xi\in K\equiv \Null(-\Lap_{a^b}-b)$, ${\br{\abs{\xi}^2}=1}$ is as in \eqref{4.20}.}

	   First, taking the inner product of the first equation in \eqref{2.13} with $\psi$, and then integrating by parts, we find
	\begin{equation*}
		\int\abs{\grad_a\psi} =  \int \kappa^2\left(|\psi|^2 - |\psi|^4\right).
	\end{equation*}
	Substituting this into the rescaled GL energy \eqref{2.12}, we get 
	\begin{equation}   \label{asymp:Elambda'}
		\mathcal{E}_r (\psi, a)  = \frac{1}{2}\del{  \frac{\kappa^2r^2}{2}\abs{\Si} + \|d a\|^2_{\cL^2} - \frac{\kappa^2}{2} \| \psi\|^4_{\cL^4} } .
	\end{equation}
	This gives 
\begin{align}\label{GNEs}
		\cE_r(0,a^b)=\frac{1}{2}\del{\frac{\kappa^2r^2}{2} +b^2}|\Si|
	\quad\text{with } b=\frac{2\pi n}{\abs{\Si}}.  
\end{align}
	Next,  
	we choose  
		$\eta:=\abs{s}^{-2}\g(t(s))$
  according to expansion \eqref{5.2}.
	 {Then $\eta\in\Om$ by construction and therefore $\inn{d \eta}{d \beta}= 0$ for any $1$-form $\beta$, see \propref{prop3.1}. Using this, the expansion $a=a^b+\abs{s}^2\eta+O(\abs{s}^4)$ (see \eqref{5.2}), together with the fact that $d a^b =b \om$, we find }
\begin{equation}		\label{asymp:a-en}
		\begin{aligned}
		\|d a\|_{\cL^2}^2 &= \|d a^b\|_{\cL^2}^2 +2 \abs{s}^2 \langle d a^n, d \eta \rangle +\abs{s}^4\|d \eta\|_{\cL^2}^2 + O(\abs{s}^6)\\
		& = b^2 \abs{\Si}  +\abs{s}^4\|d \eta\|^2_{\cL^2} + O(\abs{s}^6).
	\end{aligned}
\end{equation}
	
 	With the choice \eqref{4.20} for $\xi$, expansion \eqref{5.1} becomes $\psi=s\xi + O(\abs{s}^3)$. Plugging  this into formula 
{ 	\begin{align}
 		\label{5.15'}
 	 \frac{1}{\abs{\Si}}\norm{d\eta}_{\cL^2}^2=&-\frac{1}{4}(\br{\abs{\xi}^2}^2-\br{\abs{\xi}^4}),
 	\end{align}}
 proved in \cite[Prop.~5.5, eq.~(5.32)]{MR4049917}, {and using that $\br{\abs{\xi}^2}\equiv \frac{1}{\abs{\Si}}\int \abs{\xi}^2 =1$,}
 	 we find 
	\begin{align}\label{5.18}
		\abs{s}^4\|d \eta\|^2_{\cL^2}- \frac{\kappa^2}{2} \| \psi\|^4_{\cL^4} = \frac{{\abs{\Si}}}2\abs{s}^4\del{ \del{\frac12-\kappa^2}{\br{\abs{\xi}^4}}-\frac12}+O(\abs{s}^6).
	\end{align}
	Plugging \eqref{asymp:a-en}--\eqref{5.18}  into \eqref{asymp:Elambda'}, we conclude that         
	\begin{equation}     \label{Eexpan-s}
		   \mathcal{E}_r (\psi, a) =\cE_r(0,a^b) -\frac{{\abs{\Si}}}4  
		\abs{s}^4 \del{\del{\kappa^2-\frac12}{\br{\abs{\xi}^4}}+\frac{1}{2}} 
		+  O(\abs{s}^6).  
	\end{equation}
	This,  together with expansion \eqref{sEst2} for $s$,  gives \eqref{5.17}.
	
	
{	By definition \eqref{beta} and the {assumption $\dim K=1$}, we have ${\beta= \br{\abs{\xi}^4}}$. 
	This, together with \eqref{5.17}, the rescaling \eqref{enRel}, and the relation $\abs{\Si}=r^{-1}\abs{\Si}$,  gives \eqref{1.12}.}
\end{proof}

\section{Extension to the case $\dim K > 1$}\label{secExt}
{	In this section, we illustrate how  to extend the main existence results in \secref{sec:1} to the degenerate case, with $D = \dim K > 1$.} We also show how to extend energy estimates to this more general setting (in Section \ref{Sec7.2}).
{	\subsection{Existence theory}
	We now make some remarks on the possibility of solving the bifurcation equations \eqref{4.2} in the case when $D>1$ where $D$ is the complex dimension of $K$. The Lyapunov-Schmidt reduction detailed in section \ref{sec:4.2} carries over mutatis mutandis for $D>1$. The essential change comes in section 
	\ref{sec:4.3}, as
	we will now choose a hermitian orthonormal basis $\xi_j \,\,\, 1\leq j \leq D$  for $K$. {Define 
		$v=v_{st}=(\phi(s),\g(t))$,  
		where $s=(s_1,\ldots, s_D)$, $t=(t_1,\ldots, t_{\dim\Om})$, 
		$\phi(s) = \sum_{j=1}^D s_j \xi_j$, and  $\g(t)=\sum_{k=1}^{\dim\Om}t_k\eta_k$ as in \eqref{4.21}.
		
		Taking inner products of the bifurcation equation \eqref{4.19}}
	respectively with $\xi_j$ and $\eta_k$,  this equation can be rewritten in terms of the coordinates $s, t, r$ as
	{\begin{align}
			\label{4.26"}
			(\kappa^2r-b)s_j{+H_{\psi,j}(s,t,r)}&=0,\quad j=1,\ldots, D,\\
			\label{4.27"}
			\sum_{l=1}^{\dim \Om}B_{kl}(s)t_l+H_{\al,k}(s,t,r)&=0,\quad k=1,\ldots,\dim \Om,
		\end{align}
		where $B_{k,\ell}(s) = \langle \eta_k, |\phi(s)|^2 \eta_\ell \rangle_{\vec{L}^2}.$
		The remainders satisfy the estimates
		\begin{align}
			H_{\psi,j}(s,t,r)=& O (\abs{s} \abs{t})+O(\abs{s}^3),\label{RpsiEst"}\\
			H_{\al,k}(s,t,r)=&O(\abs{s}^2\abs{t}^2)+O(\abs{s}^4),\label{RalEst"}
		\end{align}
		and similarly for their derivatives. 
		\medskip

		{The solution of \eqref{4.27"}  is based on the following classical result of Krasnoselski that reduces, under appropriate conditions, a vector bifurcation problem to a scalar one. 
			\begin{theorem}[{\cite[Chapt.~4 ]{MR660633}}]\label{Kras}
				Let $E: \mathbb{R}^{n+1} \to \mathbb{R}^n$ be $C^2$ of the form 
				\begin{eqnarray} \label{normalform}
					E(\ell,s) &=& \ell s + H(\ell, s)
				\end{eqnarray}
				where $s \in \mathbb{R}^n, \ell \in \mathbb{R}$ and $H(\ell, s)= O(|s|^2)$ as $|s| \to 0$ uniformly in a  neighborhood of $\ell = 0$. 
				Assume that $n$ is odd; then, $(\ell,s) = (0,0)$ is a bifurcation point of $E(\ell, s)$; i.e., there are solutions $(\ell, s)$ with $s \ne 0$ of the equation $E(\ell, s)=0$
				in every neighborhood of $(\ell, s) = (0,0)$.
			\end{theorem}
			\begin{remark}
				We note that \eqref{normalform} is the standard normal form for a nonlinear problem whose linearization has $\ell = 0$ as an isolated eigenvalue of multiplicity
				$n \geq 1$. 
			\end{remark}
			We defer the elegant proof of Theorem \ref{Kras} to an Appendix~\ref{secPfKras}, and turn to applying it to our problem. }
		
		
		The first step is to consider solutions  of \eqref{4.26"}. We find that 
		\begin{lemma}\label{lemT}
			Eq.~\eqref{4.27"} has a unique  solution $t=t(r,s)$ in a small neighbourhood of $(b/\kappa^2,0)$, such that $st(s,r)$ is $C^2$.		
		\end{lemma}
		
		We momentarily defer the proof of this lemma in order to apply it to solving the full bifurcation problem. As a consequence of the lemma, insertion of  $t(s,r)$ into the higher order terms \eqref{RpsiEst"} of \eqref{4.26"} results in a biifurcation equation that is  $C^2$ which will enable us to satisfy regularity conditions of Theorem \ref{Kras}.   We now want to rewrite \eqref{4.26"} in terms of real variables by setting $s_j =x_j + i y_j$.  We also apply a constant gauge transformation to reduce $s_1$ to $s_1 = x_1$. With this gauge fix in place, we
		take $s$ to denote the coordinates, $s = (x_1, x_2, y_2, \dots x_D, y_D)$, which coordinatizes $\mathbb{R}^{2D-1}$. So in our application $n = 2D - 1$ which
		is odd so that condition of the theorem will be satisfied. With this gauge fix in place, \eqref{4.26"} is precisely in the normal form \eqref{normalform}. Indeed, setting $\ell = \kappa^2 r - b$ and substituting the solution $t=t(r,s)$ found in \lemref{lemT} to \eqref{4.26"}, the equation takes the vector form 
		\begin{eqnarray} \label{normalform'}
			\tilde{G}(\ell,s) &=& \ell s  + H(s, r)    = 0
		\end{eqnarray}
		where $H(s,r) =  O(|s|^2)$.
		Since $\tilde{G}(\ell,s)$ is $C^2$, all conditions of Theorem \ref{Kras} are satisfied. This establishes the existence of a supercritical zero. Given that, an application of the vector implicit function theorem yields the existence of a smooth branch of zeroes extending from the bifurcation point to the supercritical zero. Hence, Proposition 4.1 extends to be true for $D > 1$. 
		\medskip
		
		As a brief illustration we note how Theorem \ref{Kras}, for the case $D=1$ aligns with our earlier proof of Lemma \ref{thmD.1}. From the proof of Theorem \ref{Kras} given in Appendix \ref{secPfKras}, when $n =1$, the constraint simply reduces to $s\hat{E}(s) = 0$. So for each $s \ne 0$, $\hat{E}(s) = 0$. 
		It follows that the unique $C^2$ curve $\ell(s)$ found in the proof of Theorem \ref{Kras} determines the branch of supercritical solutions to the bifurcation equation. Applying this to \eqref{normalform'}, this is the branch $\kappa^2 r(s) -b$, determining a unique  $r(s)$ that  satisfies
		\begin{equation}\label{rEst'}
			r(s)=\frac{b}{\kappa^2}+O(\abs{s}^2).
	\end{equation}}
	\medskip

	We now turn to the proof of the \lemref{lemT}.
	
	\begin{proof}[Proof of \lemref{lemT}]
		
		We consider the quadratic form associated to the leading terms of  \eqref{4.27"},

		\begin{eqnarray} \nonumber
			\Phi(s,t) &=& \langle B(s) t, t\rangle\\  \label{t1}
			&=& \sum_{k,\ell = 1}^{\dim \Omega} t_k t_\ell B_{k \ell}(s)\\  \label{t2}
			&=& \langle \omega, |\phi(s)|^2 \omega \rangle_{\vec{L}^2}  \label{t3}
		\end{eqnarray} 
		%
		It is manifest from this that $\Phi$ is real valued and, away from $s=0$,  defined, differentiable and 
		positive definite as a quadratic form in ${t}$. 
		\medskip		
		
		To understand $\Phi$ in a neighborhood of $s=0$ we do a (real) algebro-geometric blow-up of $s$-space at the origin.
		This construction replaces the origin by a compact hypersurface diffeomorphic to the projective space of all local directions through the origin. More precisely, in our situation, the blow-up is defined as follows. Let $n = 2D-1$ and 
		consider the product space 
		$\mathbb{R}^n \times \mathbb{P}^{n-1}$ with $\mathbb{R}^n$ coordinatized by the affine coordinates $\vec{w} = (\vec{x}, \vec{y})$ and 
		$\mathbb{P}^{n-1}$ coordinatized by the {\em homogeneous} coordinates $\vec{z} = [z_1, \dots z_n]$.  The blow-up of $\mathbb{R}^n$ at the 
		origin, which we'll denote by $M$, is the closed subset of $\mathbb{R}^n \times \mathbb{P}^{n-1}$ defined by the equations 
		$\{ w_i z_j - w_j z_i  = 0\,\, | \,\, i,j = 1, \dots, n \}$. One has a natural map $\pi : M \to \mathbb{R}^n$ induced by projection onto the first factor. It is 
		fairly straightforward to check that away from $\vec{w} = 0$, $\pi$ is a diffeomorphism. However, $\pi^{-1}(0) \simeq \mathbb{P}^{n-1}$; it consists of all
		points of the form $0 \times [z_1, \dots z_n]$.  Now to see that points of $\pi^{-1}(0)$ are in 1:1 correspondence with the set of lines through the origin in $\mathbb{R}^n$, note that a line $L$ through the origin is parametrically given by $w_i = c_i \sigma$ where the $c_i$ are not all zero. Then consider the lift of this line to  $\tilde{L} = \pi^{-1}(L - 0)$ in $M - \pi^{-1}(0)$ whose parametrization is $w_i = c_i \sigma, z_i = c_i \sigma$. However, since $z_i$ are homogeneous coordinates, one may as well take this parametrization to be $w_i = c_i \sigma, z_i = c_i $. These equations are well-defined for $\sigma = 0$, defining the closure of $\tilde{L} $ in $M$ which meets  $\pi^{-1}(0)$ in the point $[c_1, \dots, c_n] \in  \mathbb{P}^{n-1}$. This defines the mapping, $L \to [c_1, \dots, c_n]$ that gives 
		the 1:1 correspondence between lines through the origin and points of $\pi^{-1}(0)$. For further explanation and details we refer the reader to 
		\cite{MR0463157} on which the above is based.
		\smallskip
		
		Applying this to $\Phi$, we set $A_{jm} = \int_{F_\Sigma} \xi_m \bar{\xi}_j \omega \wedge * \omega = \mu_{jm} + i \nu_{jm}$. Then one has a representation of $\Phi$ as 
		\begin{eqnarray} \nonumber
			\Phi(s,t) &=& \sum_{1 \leq j \leq m \leq D} \left(  A_{jm} s_j \bar{s}_m + \bar{A}_{jm} \bar{s}_j s_m \right)\\  \label{t9}
			&=& 2 \mu_{jm} (x_j x_m + y_j y_m) - 2 \nu_{jm} (y_j x_m - x_j y_m).
		\end{eqnarray}
		To lift $\Phi$ to $M$ we set $\vec{c} = [a_1, a_2, b_2, \dots a_D, b_D]$. Applying the corresponding parametrization in $\Phi$ one has
		\begin{eqnarray} \label{t10}		
			\Phi(s,t)   &=& 2 \sigma^2 \left(\mu_{jm} (a_j a_m + b_j b_m) -  \nu_{jm} (b_j a_m - a_j b_m)\right).
		\end{eqnarray}
		Under the blow-up change of coordinates to $(\vec{c}, \sigma, r)$, the higher order terms scale similarly but involve powers of $\sigma$ greater than 2. 
		Consequently, the terms in $(1/\sigma^2 )$ times the quadratic form associated to \eqref{4.27"} are polynomial and so these equations are clearly  $C^2$ in the blow-up coordinates. They have 	a leading term that is independent of $\sigma$. Hence the positivity seen
		in \eqref{t3} is inherited by $(1/\sigma^2) \Phi(s,t)$ from \eqref{t10}. But the latter is defined just in terms of  $\vec{c}$ which are coordinates on a compact projective space. So by continuity, $(1/\sigma^2) \Phi(s,t)$ realizes its infimum at some point in $\mathbb{P}^{n-1}$. But 
		this infimum is bounded away from zero. So $B(s)$ is globally invertible and, by the implicit function theorem, $t$ is smoothly defined in the blow-up coordinates and $t$ is $O(\sigma^2)$ in these coordinates and so vanishes in a continuous fashion as $\sigma \to 0$. Hence $t$ on the blowup $M$ can be pushed forward under $\pi$ to $t(s,r)$ in the original $s$-space. Similarly for the first derivatives of $t(s,r)$ and the second derivatives of $st(s,r)$. 
	\end{proof}

	\subsection{Energy estimates}\label{Sec7.2}

	It is possible to extend the energy estimate proved in \corref{cor1.2} to the degenerate case. 
	{In general, if one drops the second part in condition  \eqref{b0cond},}
	then, for the solution $(\psi_{s(r)}, a_{s(r)})$ constructed in \eqref{1.7}, we have, 
	\begin{equation}
		\label{1.12gen}
		{\cE(\psi_{s(r)},a_{s(r)},h_r)\ge\cE(0,a^{b_r},h_r)}{}-\frac{{\abs{\Si}_r}}4 \frac{\abs{\kappa^2-{b_r}}^2}{(\kappa^2-\tfrac{1}{2})\beta(r)+\tfrac{1}{2}}+O(\abs{\kappa^2-{b_r}}^3)
	\end{equation}
	for $\kappa \ge 1/\sqrt2$, and the opposite inequality for $\kappa < 1/\sqrt2$. 
	
	Indeed,	for $\dim K>1$, we have  $\beta(r)\le \br{\abs{\xi}^4}$ for all $\phi\in K,\,\br{\abs{\phi}^2}=1$. This, together with \eqref{5.17}, gives \eqref{1.12gen} depending on the sign of $\kappa^2-\frac12$. 
	
	One can also introduce the `upper Abrikosov function', 
	\begin{align}\label{beta+}
		\beta_+(r):=	{\max\Set{{\br{\abs{\xi}^4}}:\xi\in K(r),\br{\abs{\xi}^2}=1}.}
	\end{align}
	which leads to the upper energy bound, for $\kappa^2\ge1/2$,
	\begin{equation}
		\label{1.12+}
		{\cE(\psi_{s(r)},a_{s(r)},h_r)\le\cE(0,a^{b_r},h_r)}{}-\frac{{\abs{\Si}_r}}4 \frac{\abs{\kappa^2-{b_r}}^2}{(\kappa^2-\tfrac{1}{2})\beta_+(r)+\tfrac{1}{2}}+O(\abs{\kappa^2-{b_r}}^3).
	\end{equation}
	For $\kappa_{c+}(r):=\sqrt{\frac{1}{2}\del{1-\frac{1}{\beta_+(r)}}}$, estimate \eqref{1.12+} ensures that
	$\cE(\psi_{s(r)},a_{s(r)},h_r)<\cE(0,a^{b_r},h_r)$ for $\kappa>\kappa_{c+}(r)$.

\section*{Acknowledgment}
J. Zhang is supported by Danish National Research Foundation grant CPH-GEOTOP-DNRF151, The Niels Bohr Grant of The Royal Danish Academy of Sciences and Letters, and NSERC Grant No. NA7901.  The research of IMS is supported in part by NSERC Grant No. NA7901.

\section*{Declarations}
\begin{itemize}
	\item Conflict of interest: The Authors have no conflicts of interest to declare that are relevant to the content of this article.
				\item Data availability: Data sharing is not applicable to this article as no datasets were generated or analysed during the current study.
\end{itemize}

\appendix

\section{Standard definitions and notation} \label{sec:A}

We sketch very briefly some notions and definitions relevant for us .
To fix ideas, in what follows,  $E$ is a line bundle over a Riemann surface $\Si$,  with fibers isomorphic to $\C$ and the gauge group  $G=U(1)$.

\subsection{Connections}\label{sec:A.1}
{A connection (gauge field), $a$, is a real-valued one-form with certain transformation properties. 
For a connection $a$, defines the  covariant derivative, $\n_{a}$, on sections of $E$, which can be written locally as  $\n_{a} \psi=\grad \psi -i  a\psi$. 
Here $\grad$ is a fixed connection (say, the Levi-Civita one).}

A connection $a$ defines the curvature two-form  $F_{a}:= da$. 
A connection $a$ on $E$ is said to be a {\it constant 
	curvature connection} if its  curvature is of the form 
\begin{align} \label{const-curv-cond} da=b\om,   \end{align} 
for some $b\in \R$, where $\om$ is the standard hyperbolic $2$-form on $\Si$.

By \cite[{Lem. 3.2}]{MR4049917}, a constant 
curvature connection,  $a$, solves the ({\it static}) Maxwell equation 
\begin{equation}   \label{Maxw-eq}
	d^* d a = 0.
\end{equation}

Let $x^i,\,i=1,2$ be a local coordinate on $\Si$.
With the summation convention understood, we can write a connection (gauge field), $a$,
and the  covariant derivative, $\n_a$,  as   $a =a_{i} dx^i $ and  $\n_a= \n_i d x^i$, where  
$\n_j := \p_i -ia_j$.   
Let $d$ be the exterior derivative on $\Si$.
Then, for a section $\psi$ and 1-form $B =B_{i} dx^i$,   \[d_{a} \psi =\n_{i} \psi dx^i ,\ d B =\di_i B_j dx^i \wedge dx^j,\ d^* B =-\di_i B_i\] 
and the curvature, $F_a=d a$, of $a$ is given by  
\begin{align}\label{curv-expr}    &F_a= \big[\nabla_{i }, \nabla_{j } \big]d x^i \wedge d x^j,\ \quad 
	\big[\nabla_{i }, \nabla_{j } \big]
	=\p_i a_j-  \p_j a_i.
\end{align}

\subsection{{Automorphy factor}}\label{sec:2.2}
Let $\Si=\Hb/\Ga$ be a non-compact Riemann surface of the form \eqref{2.1}.
Let
$\pi_1(\Si)$ be the first fundamental group of $\Si$,
which we identify  with $\Ga$.
A map
$$\rho(\g,z):\pi_1(\Si)\times \Hb\cong \Ga\times \Hb\to U(1)$$ 
is called an \textit{automorphy factor} if it satisfies the
important co-cycle condition,
\begin{equation}
	\label{2.17}
	\rho(\g\g',z)=\rho(\g,\g'z)\rho(\g',z)\quad \del{\gamma,\g'\in\Ga,\, z\in \Hb}.
\end{equation}
For fixed $\g\in\Ga$, we often write $\rho_\g(z)\equiv \rho(\g,z)$ 
as a function from $\Hb$ to $U(1)$.

For every $b\in\Rb$, we fix a canonical choice of automorphy factor with weight $b$ as
\begin{equation}
	\label{2.18}
	\rho_b(\g,z)\equiv\rho_{b,\g}(z):=\del{\frac{c z+d}{c\bar z+d}}^{b}\quad\del{\gamma=\begin{pmatrix}a&b\\c&d\end{pmatrix}\in SL(2,\Rb)}.
\end{equation}
For results in this paper, we understood that 
$b\in\Z$, but \eqref{2.18}
can also be defined for non-integer values of $b$
with branch cuts.

The choice \eqref{2.18} is customary in the study of Maass forms in number theory,
see e.g. \cite[{Sec.2.1}]{MR1431508}.
In \thmref{thm:aut-factsEn}, we prove this choice of $\rho(\g,z)$ satisfies
the co-cycle condition \eqref{2.17}.
Note that $\rho(\g,z)$ remains bounded
as $z$ approaches the boundary $\Rb\cup\Set{\infty}$ of $\Hb$,
which contains the cusps of $\Si$.

Recall that the universal cover of the Riemann surface
$\Si$ is the Poincar\'e half-plane $\Hb$. 
Then using the co-cycle condition \eqref{2.17},
one can define a line bundle $E_\rho$ as 
\begin{equation}\label{erho}
	E_\rho := (\Hb \times \Cb)/\rho, \text{ with the action } \rho_\g:(z,\psi)\mapsto(\g z,\rho_\g(z)\psi).
\end{equation}
In fact, \eqref{erho} defines an one-to-one correspondence, $\rho\leftrightarrow E_\rho$, between (equivalence classes of)  automorphy maps for $\Si$ and line $U(1)$-bundles over $\Si$. See
\cite{MR0132828} for details.

Through the correspondence \eqref{erho}, one can define a topological degree for a unitary line bundle over $\Si$, which by the Chern-Weil correspondence is equal to the Chern number, or the degree of $E$,  as follows.  
Suppose $E=E_\rho$ as in \eqref{erho} for an automorphy factor 
$\rho$ satisfying \eqref{2.17}.
Suppose the Fuchsian group  $\Ga$ has genus $g$ and $m$ cusps, with no elliptic points. Then $\Ga$
is generated by $2g$ hyperbolic transforms,
$A_1,B_1,\ldots,A_g,B_g$, and $m$ parabolic transforms $S_1,\ldots,S_m$,
satisfying the relation
\begin{align} \label{gener-relat} A_1 B_1A_1^{-1} B_1^{-1} \dots A_g B_gA_g^{-1} B_g^{-1}S_1 \dots S_m=\one.
\end{align}

Now, for every $\g\in\Ga$, define a map 
$\si=\si_\g:\Hb\to\Rb$
by the relation $e^{i\si_{\g}(z)}:=\rho(\g,z)$. 
The first Chern number $c_1(\si)$ is defined by
\begin{align} \label{c1g}  c_1(\si):=&\frac{1}{2\pi}\sum_1^g [\si_{A_{i}}(v_i) - \si_{A_{i}}(B_{i}^{-1}v_{i}) +\si_{B_{i}}(B_{i}^{-1}v_{i})-\si_{B_{i}}(A_i B_{i}^{-1}v_{i})],
\end{align} 
where $v_{i}:=B_i A_{i}^{-1}B_{i}^{-1} C_{i+1}\dots C_g z_0= A_{i}^{-1} C_{i}\dots C_g z_0$, with $C_{i}:= A_{i} B_i A_{i}^{-1}B_{i}^{-1} $.   
By \cite[{Theorem 2A}]{MR0457787},
$c_1(\si)$ is independent of $z_0$, and takes values in $\Zb$.

Let $E$ and $E'$ be the line bundles corresponding to two automorphy factors $e^{i\si},\,e^{i\si'}$ respectively.
Then we say that $E$ and $E'$ are equivalent  if and only if
$c_1(\si)=c_1(\si').$

Hence, by the one-to-one correspondence \eqref{2.17},
$c_1(\si)$  is a  topological invariant for a unitary line
bundle $E=E_\rho$ over $\Si=\Hb/\Ga$. It turns out (see \cite{Gunning62}) that $c_1(\si)$ is equal to the topological degree of $E$:
\begin{equation}\label{deg}
	c_1(\si)=\deg E_{e^{i\si}}.
\end{equation}
(Recall that $\deg E$ is defined through multiplicity of zeroes of sections.)

\subsection{Equivariant States}\label{sec:2.4}

\begin{definition}\label{defn2.1}
	Let $\Ga$ be a Fuchsian group.
	A (function, $1$-form)-pair $(\Psi,A):\Hb\to \Cb\times \Rb^2$ is called an equivariant state w.r.t.~$\Ga$
	if and only if for all $\g\in\Ga$, $z\in\Hb$, the following relations hold:
	\begin{align}
		\label{2.15'}
		\g^*\Psi(z)&=\rho(\g,z)\Psi(z),\\
		\label{2.16'}
		\g^*A(z)&=A(z)+i\rho(\g,z)^{-1}d\rho(\g,z),
	\end{align}
	where $\g^*$ is the pull-back by the M\"obius transform \eqref{2.2}
	associated to $\g\in\Ga$, and $\rho(\g,z)$ 
	is an {automorphy factor},
	satisfying the {co-cycle} condition \eqref{2.17}
\end{definition}

\begin{remark}
	For a standard lattice $\cL\subset \Cb\cong \Rb^2$,  equivariant solutions to \eqref{GL}
	are known as the Abrikosov lattices,
	predicted by A.A. Abrikosov 
	in 1957 \cite{Abrikosov}. (This discovery was recognized by a Nobel prize.)
	These are ground state solutions to \eqref{GL}
	on the flat torus
	\begin{equation}\label{Tb}
		\T=\Cb/\cL,
	\end{equation}
	which is a compact Riemann surface of genus $1$.
	Existence and stability theory of solutions to  \eqref{GL}
	on $\T$ are studied in \cite{MR2560758,MR2904275,MR3123370,MR3758428}.
	Physically, these solutions correspond to regular arrays of vortices as seen
	in Type II superconductors.
	
	In \eqref{Tb},  the lattice $\cL$ acts on the complex plane
	by translation, and the action is commutative. 
	In comparison, with the background geometry \eqref{2.1}, the action of $\Ga$
	on the Poincar\'e half plane $\Hb$ by M\"obius transforms
	is in general non-commutative.
	In this sense, equivariant solutions to \eqref{GL}
	on \eqref{2.1} are non-commutative
	generalizations of the Abrikosov lattice.
\end{remark}


 Equivariant states $(\Psi,A)$ are in one-to-one correspondence with (section, connection)-pairs $(\psi,a)$ on the unitary complex 
 line bundle $E\to\Si$ through the explicit correspondence \eqref{erho}. 
 One can restrict $(\Psi,A)$ to a fundamental domain $F\subset\Hb$ of $\Ga$, with \eqref{2.15'}--\eqref{2.16'} considered as boundary conditions. 
 Although fundamental domain is not unique, every two fundamental domains $F,F'\subset \Hb$ of $\Ga$
 are related  by a M\"obius transform  $\g\in\Ga$, which sends $F\to F'$.
 Hence, 
 if  $(\Psi,A)$ is an equivariant state on $F$  satisfying \eqref{2.15'}--\eqref{2.16'}, 
 then the change of identification $F\to F'$ amounts to a gauge transform
 of the form 
 \begin{align}
 	\label{2.14}
 	(\Psi,A)\longrightarrow (\rho\Psi,A+i\rho^{-1}d\rho).
 \end{align}
 Since \eqref{GL} is invariant under \eqref{2.14}, such change
 has no consequence when one studies  the solution to \eqref{GL}.

 
The important correspondence formulated above allows us to interpret problems posed in $X^s$ in terms of equivariant states. For example, given an automorphy factor $\rho_b$ with weight $b$, let   $L^2(\Hb/\Ga)\equiv L^2(\Hb/\Ga,\rho_b)$ be the space of square integrable equivariant function satisfying \eqref{2.15'} with $\rho=\rho_b$ (see e.g.~\cite[Sect. 2.1]{MR1431508} for details). 

The equivariant states picture is useful for explicit computations. Below, we obtain an explicit description of the null space $K$ from \eqref{3.9} in terms of equivariant functions in $L^2(\Hb/\Ga)$ and prove inclusion \eqref{inclusion2}, using similar symmetrization argument as in  \cite{MR1942691}, which studies the same problem
with $b=0$. 
This method was introduced by Selberg in 1950s,
and is well-known to the number theorists.  

\begin{proposition}
	\label{prop3.2}
	Let $\Si=\Hb/\Ga$ be a non-compact Riemann surface with $m$ cusps and no elliptic points,
	with $\cS(\Si)\ne\emptyset$.
	Then the solution space to \eqref{3.2} in $L^2(\Hb/\Ga)$ is spanned by the vectors
	\begin{align}
		\label{3.10}
		\xi_i(z)&:=\sum_{\g\in\Ga/\Ga_i} \Im(\g_i\g z)^{b}e^{\g_i\g 2\pi i z}\rho_b(\g,z)^{-1},\quad i=1,\ldots,m
	\end{align}
	where  each
	$\Ga_i:=Stab(c_i,\Ga)$
	denotes the stabilizer of a distinct cusp $c_i$ of $\Ga$,
	and $\g_i\in SL(2,\Rb)$ is a scaling matrix of $c_i$.
	
	Moreover, there holds the asymptotics
	\begin{equation}
		\label{3.10'}
		\abs{\xi_i(\g_i^{-1}z)}=O(e^{-2\pi y})\quad (y\to\infty),
	\end{equation}
	which implies that $\xi_i(z)$ decays exponentially fast as $z$ approaches 
	the cusp $c_i$. 
\end{proposition}

\begin{proof}

	Suppose a function
	$\phi:\Hb\to\Cb$ given by  
	$\phi(x,y)\equiv \phi(x+iy)$ satisfies 
	\begin{equation}
		\label{3.11}
		\phi(x+1,y)=\phi(x,y)\quad \text{ for every }y>0.
	\end{equation}
	Then for every cusp $c_i$, we can form the Poincar\'e series 
	\begin{equation}
		\label{3.12}
		E_{c_i,\phi}(z):=\sum_{\g\in\Ga/\Ga_i} \phi(\g_i\g z)\rho^{-1}(\g,z).
	\end{equation}
	This series \eqref{3.1} converges absolutely if $\phi$ satisfies 
	certain growth condition in $y$.

	The Poincar\'e series has two important properties:
	\begin{enumerate}[label=(\alph*)]
		\item By construction,
		$E_{c_i,\phi}$ satisfies the equivariance condition \eqref{2.15'};
		\item If $L$ is an invariant operator acting on $\Hb/\Ga$
		{in the sense that \begin{equation}\label{3.8}
				\Delta_{a^b}(\g^*\phi)=\rho(\g,z)(-\Lap_{a^b}\phi)
			\end{equation}for every $\g\in\Ga,\,\phi\in \cH^s$,} and $L\phi=\l\phi$ for some $\l\in\C$,
		then $LE_{c_i,\phi}=\l E_{c_i,\phi}$.
	\end{enumerate}
	In view of these properties, to solve the eigenvalue problem \eqref{3.2} in $L^2(\Hb/\Ga)$,
	we first solve the following problem for some $\phi\in C^2(\Hb,\Cb)$:
	\begin{align}
		(-\Lap_{a^b}-b)\phi &=0\label{evp1},\\
		\phi(x+1,y)&=\phi(x,y),\quad \lim_{y\to\infty}\phi(x,y)=0.\label{evp2}
	\end{align}

	In the boundary condition \eqref{evp2}, the periodicity in $x$
	ensures that  the symmetrization is well-defined,
	c.f. \eqref{3.11}.
	The decay property in $y$ is required 
	to ensure convergence of the Poincar\'e series \eqref{3.12}.
	
	The solution to \eqref{evp1}--\eqref{evp2} is proportional to
	the function
	$$\phi:=y^be^{-2\pi y}e^{i2\pi x}.$$
	This is calculated 
	in e.g. \cite{MR776146,MR870891}.
	Forming Poincar\'e series \eqref{3.12} w.r.t. each of the cusps of $\Si$ 
	gives \eqref{3.10}. By construction, \eqref{3.10} are equivariant solution to \eqref{3.2}.
	
	By the classical results for the Fourier analysis of Poincar\'e series,
	e.g. \cite[{Sect. 2}]{MR750670}, there holds the Fourier expansion 
	\begin{equation}
		\label{3.13}
		\xi_i(\g_i^{-1}z)=\sum_{k\ne0}A_kW_{b\sgn (k),b-\tfrac{1}{2}}(4\pi\abs{n}y)e^{2k\pi i x}.
	\end{equation}
	Here $W_{\beta,\mu}(y)=O(e^{-y/2})$ is the Whittaker function,
	a decaying solution to the ODE
	$$W''(y)+\del{-\frac{1}{4}+\frac{\beta}{y}+\frac{1/4-\mu^2}{y^2}}W(y)=0.$$
	This 
	$W_{\beta,\mu}(y)$ can be expressed in terms of the modified Bessel function of the second kind
	\cite{MR1773820}.
	Expansion \eqref{3.13} implies the asymptotics \eqref{3.10'}.
	
\end{proof}

\begin{remark}
	Estimates for the coefficients $A_k$ in \eqref{3.13} are of significant interest in number theory, and have been obtained in \cite{MR870736,MR931205}.
\end{remark}


\begin{proof}[Proof of \eqref{inclusion2}]
	We use the same symmetrization method as in the proof of \propref{prop3.2}.
	First, we seek $C^2(\Hb,\Cb)$-solutions to the eigenvalue problem 
	\begin{align}
		(-\Lap_{a^b}-\l)\psi &=0\label{evp1'},\\
		\psi(x+1,y)&=\psi(x,y).\label{evp2'}
	\end{align}
	Notice here we do not require the decay condition in $y$, c.f. \eqref{evp2}.
	The problem \eqref{evp1'}--\eqref{evp2'} has 
	has the following family of solutions:
	\begin{equation}\label{phik}
		\phi_k:=y^{1/2-ik}\quad (k\ge 0).
	\end{equation}
	These $\phi_k$'s are the generalized eigenfunctions, 
	which correspond  to the spectral points 
	\begin{equation}\label{lamk}
		\lam_k:=k^2 + \frac14+ b^2\quad (k\ge 0).
	\end{equation}
	
	Next, fix a cusp $c_i$ of $\Si$ and some $k\ge0$.
	For $n=1,2,\ldots,$ let $u_n:=\chi_n\phi_k$, where $\chi_n\in C_c^2(\Rb_{>0}, \Rb_{\ge0})$
	is a sequence of cut-off functions with 
	\begin{equation}\label{chi}
		\chi_n(y)\equiv \left\{\begin{aligned}
				&1\quad (n\le y \le n+1),\\
				&0\quad (0<n-\tfrac{1}{n}\le y\text{ or } y\ge n+\tfrac{1}{n}).
			\end{aligned} \right.
	\end{equation}

	Now, we form
	the Poincar\'e series \eqref{3.12} w.r.t. $c_i$ and $u_n$.
	Since $u_n$ has compact support, $E_{c_i,u_n}$ converges 
	absolutely for every $n$. Moreover, each $E_{c_i,u_n}$
	is an incomplete  Eisenstein series, which is 
	bounded and satisfies the equivariant condition \eqref{2.15'}.
	(See \cite[{Sec. 3.2}]{MR1942691} for a discussion.)
	Hence $E_{c_i,u_n}\in \cL^2$ and satisfies \eqref{evp1}
	except for on the two bounded strips $\Set{(n-1)/n<y<n}$ and $\Set{n<y<(n+1)/n}$. 
	
	Finally, let $$\tilde u_n:= \norm{E_{c_i,u_n}}_{L^2}^{-1}E_{c_i,u_n}.$$
	Then $\tilde u_n\in L^2(\Hb/\Ga)$ forms a Weyl sequence for $-\Lap_{a^b}$ and $\l_k$,
	c.f. \eqref{weyl}. This shows that $\l_k\in\si_{\rm ess} (-\Lap_{a^b})$.
	Varying $k$, we find $$[\tfrac{1}{4}+b^2)\subset\si_{\rm ess} (-\Lap_{a^b}).$$
\end{proof}


\section{Classification of  $U(1)$-Automorphy Factors and Constant Curvature $U(1)$-Connections} \label{sec:aut-fact-ccc}

In this section, we reproduce with minor modifications results of \cite{MR4049917} on classification of automorphy factors and constant curvature connections.  

Recall that a character of $\G$ is a homomorphism $\chi: \G  \rightarrow U(1)$. For other standard definitions, see  Appendix \ref{sec:A}.

\begin{theorem}[classification of automorphy factor] \label{thm:aut-factsEn}
	For any $\beta\in\Rb$,  the map $\rho_{\beta}: PGL(2, \R)\times \bH\ra U(1)$, given by 
	\begin{align}\label{n-automor}
		\rho_{\beta}(\g, z) & = \biggr[\frac{c  +d}{c \bar z + d}\biggl] ^{\beta },\quad \g=\left(\begin{smallmatrix}a&b\\c&d\end{smallmatrix}\right)  \in PGL(2, \R), 
	\end{align}
	satisfies the co-cycle condition \begin{align*}
		\rho_\beta(s \cdot t, z) &=  \rho_\beta(s, t \cdot z)\rho_\beta(t,z),\quad \forall s, t \in PGL(2, \R).
	\end{align*}   

	Consequently,  for any Fuchsian group $\G$ s.t. $\Si:=\bH/\G$ has finite area and $b:=\frac{2\pi n }{|\Si|}$, and any character $\chi: \G  \rightarrow U(1)$,  
	the map $\rho_{b, \chi}: \G\times \bH\ra U(1)$ given by   
	\begin{align}\label{n-sig-automor}
		\rho_{b, \chi}(\g, z) & = \chi(\g) \rho_{b}(\g, z), \quad
		\g \in \G, 
	\end{align}
	is also an automorphy factor. 
	Two automorphy factors related as in \eqref{n-sig-automor} are said to be equivalent.
	
	Let $n\in \Z$ and $b:=\frac{2\pi n }{|\Si|}$.  
	Then, the first Chern number of $\rho_{b, \chi}$ is $c_1(\rho_{b, \chi})=n$. Hence  any  automorphy factor $\rho:\G\times \bH\ra U(1)$ of degree $n$ is equivalent to $\rho_{b}$ as in \eqref{2.18}. 
\end{theorem} 
\begin{proof}Let 
	$s=\left(\begin{smallmatrix}a&b\\c&d\end{smallmatrix}\right),\ t=\left(\begin{smallmatrix}e&f\\g&h\end{smallmatrix}\right)  \in \G$. Using \eqref{n-automor}, we compute \textit{for any $\beta$}, 
	\begin{align*}
		\rho_\beta(s \cdot t, z) &= 
		\biggr[\frac{(ce + dg)z +(cf + dh)}{(ce +dg)\overline{z} + (cf + dh)}\biggl] ^{\beta} \\
		&= 
		\biggr[\frac{c(e z + f)  +d(g z + h) }{(c(e\overline{z} + f) + d(g\overline{z} + h)}\biggl] ^{\beta} \\
		&= 
		\biggr[\frac{c\frac{e z + f}{g z + h} +d}{c\frac{e\overline{z} + f}{g\overline{z} + h} + d}\biggl]^{\beta}  \biggr[\frac{g z + h}{g\overline{z} + h}\biggl]  ^{\beta} \\
		&= \rho_\beta(s, t \cdot z)\rho_\beta(t,z).
	\end{align*}
	Using the formula for  the Chern class, $c_1(\rho)$, of  a co-cycle $\rho$ (see \cite{MR0457787}, Theorem 2a), we compute $c_1(\rho_\beta)=n$, provided $\beta=\frac{2\pi n }{\abs{\Si}}$, where $\abs{\Si}$ is the area of $\Si$ w.r.t. to the standard hyperbolic
	metric.
\end{proof}

\begin{theorem}[classification of constant curvature connections]\label{thmB.2}
	For any $b\in\Rb$, the connection 
	\begin{align}\label{Ab'} A^b =b y^{-1}dx,\end{align} 
	on the trivial line bundle $\tilde E := \bH\times \C$  
	
	(a) has a constant curvature with respect to the standard hyperbolic area form on $\bH$;   
	
	(b) is equivariant with respect to the automorphy factor \eqref{n-sig-automor} for any Fuchsian group $\G$ and any character $\chi: \G  \rightarrow U(1)$;  
	
	(c)  is  unique (as a connection on $\tilde E := \bH\times \C$) up to gauge transformations \eqref{gaugetransf}. \end{theorem}

\begin{remark}\label{rem:af-gauge-equiv} 
	The description in the above theorem does not depend at all on the complex structure of the underlying Riemann surface.  
	Hence, if $E_{n,\chi}$ is the unitary line bundle
	corresponding to the automorphy factors \eqref{n-sig-automor}, then
	the projection of $A^b$ to $E_{n, \chi}$ gives the distinguished connection $a^{b, \chi}$ on $E_{n, \chi}$.
\end{remark}

\begin{proof} [Proof of Theorem \ref{thmB.2}]
	We begin with some preliminary constructions. 
	We consider the trivial bundle, $\tilde E:= \bH \times \mathbb{C}$ with the standard complex structure on $\bH$ associated to the standard hyperbolic metric $\tilde h=  (\im z)^{-2} |dz|^2$. 
	
	Since we work here on a global product space, it is natural to take the fiber metric to be induced from the metric on the base. So we take the metric on the fiber $\mathbb{C}$ over the point $z \in \bH$ to be $k_z  = (\im z)^{-2} |dw|^2$, where $w$ is the coordinate on the fiber $\mathbb{C}_z$.
	
	Let the connection $A$ be given by $A:=A_1 d x_1+A_2d x_2$. We decompose the covariant derivative  $\nabla_A$ into $(1, 0)$ and $(0, 1)$ parts as $\nabla_A=  \partial_A' +  {\partial_A''}$, where $\partial_A'$ and $\partial_A''$ are defined in \eqref{pA}--\eqref{compl-conn}.

	Recall from Section \ref{sec:3.3} that, in terms of $A_c$, the curvature is given by   $F_A=2 \re \bar \p A_c$. Moreover, if $A_c$ satisfies the equivariance relation $s^* \bar A_c = \bar A_c  - i \overline{\partial}\tilde f_s$, then $A$ satisfies $s^* A = A  + d f_s$, with $f_s$ satisfying $d f_s:= 2 \im \overline{\partial}\tilde f_s$.
	
	According to \eqref{compl-conn}, the  complexification of the connection $A^b$ given in the theorem   is 
	\[A_c^b = b \frac{ 1}{2\im(z)}d z.\]

	In the remaining of this proof, we omit the superindex $b$ in $A^b$ and $A_c^b$.

	\textit{Proof of constant curvature.} 
	Using that $\omega = \frac{i}{2}\Im(z)^{-2} dz\wedge d\overline{z}$, we find
	\begin{align*}
		\bar \p  A_c &= \frac{\partial}{\partial \bar z} \frac{i b}{z - \overline{z}}d\bar z\wedge d z 
		=  \frac{- i b}{(z - \overline{z})^2}dz\wedge d\overline{z} = \frac{i b}{4\Im(z)^2}dz\wedge d\overline{z} = \frac{ b}{2}\omega.
	\end{align*}
	Since $F_A= 2 \re \bar \p A_c$, this gives the desired result.

	\textit{Proof of uniqueness.}
	If $A$ and $B$, satisfy $dB = dA$ then $d(A-B) = 0$. It follows from the simple connectedness of $\mathbb{H}$ 
	that $A-B = df$ for some function $f: \mathbb{H} \rightarrow \mathbb{R}$ and $f$ is unique up to an additive constant.  So we can map $B$ to $A$ through a suitable reparametrization.  
	This completes the proof.
	
	\textit{Proof of equivariance.}  For a generic isometry $s(z) =  \frac{\al z + \bet}{\g z +\delta}$, we have $\frac{\partial s(z)}{\partial z}=(\g z +\delta)^{-2}$ and $\Im(s(z))=\frac{\Im(z)}{ |\g z + \delta|^2} $, which gives 
	\begin{align*}
		s^* \bar A_c &= \frac{ b}{2\Im(s(z))}\overline{\frac{\partial s(z)}{\partial z}}d\overline{z}= \frac{ k |\g z + \delta|^2}{2\Im(z)} \frac{d\overline{z}}{(\g \overline{z} + \delta)^2}  \\
		&= \frac{ b}{2\Im(z)}\frac{\g z + \delta}{\g \overline{z} + \delta}d\overline{z}= \frac{ b}{2\Im(z)}d\overline{z} + \frac{ b (\g z - \g \overline{z})}{2\Im(z)(\g \overline{z} + \delta)} d\overline{z} \\
		& =\bar  A_c + \frac{i b \g}{\g \overline{z} +\delta} d\overline{z}  =: \bar A_c  +  \overline{\partial}\tilde f_s,\end{align*} 
	where $\tilde f_s$ is the function defined by the last relation, i.e. $ \overline{\partial}\tilde f_s=\frac{i b \g }{\g \overline{z} +\delta}d\overline{z}.$
	Solving this equation, we find \[\tilde{f_s} = b\text{ln}(\g \overline{z} + \delta)+c_s.\] 
	Now, we define $f_s := 2\re \tilde{f_s}$ and use that $\re (\bar\p\tilde{f_s}) =d (\re \tilde{f_s})$ (as can be checked by the direct computation: $\frac1b \re (\bar\p\tilde{f_s}) = - \frac{\g^2 x_2}{|\g z +\delta|^2} dx_1 + \frac{\g (\g x_1 +\delta)}{|\g z +\delta|^2} d x_2$) to obtain $s^* A =  A  + d f_s$, with
	\begin{align}\label{gs}
		& f_s = 2\Re(\tilde{f_s}) = i b \text{ln}\biggr[\frac{\g \overline{z} +\delta}{\g z + \delta}\biggl] +c_s,\ 
		\rho(s, z) = e^{if_s(z)} =e^{i c_s}  \biggr[\frac{\g \overline{z} +\delta}{\g z + \delta}\biggl] ^{-b}.
	\end{align}
	Here note that since $\mathbb{H}$ is the upper half plane, the complex logarithm is well defined and $\g z + d$ is always non zero.   \end{proof}

\begin{remark} 
The function $\tilde f_s (z)$ appearing above gives the character  $\tilde \rho (z) = e^{\tilde f_s (z)}$, which  is now $\mathbb{C}^*$-valued instead of $U(1)$-valued. 
\end{remark}
\begin{remark}

	$A^b$ is $\mathbb{R}$ - linear while $A^c$ is naturally $\mathbb{C}$ - linear. It is natural to ask what the action of $i$ on $A^n_c$ does when we map back to $A^n$. A simple calculation shows that $iA^b_c = 
	i \frac b2 \frac{ 1}{\im(z)}d z$ is mapped to $b y^{-1}dy$, which is flat. It turns out that the complex action of $i$ induces a rotation into the space of flat connections. 
\end{remark}


%

\section{Chern-Weil Correspondence} \label{sec:flux-quant-pf}

\begin{proof}[Proof of \thmref{thm2.1}] 
	The argument below, due to D. Chouchkov, is streamlined from classical results from \cite[Chap. II.4]{MR0457787}.
	

%
	{Let $F_\Si\subset \Hb$ be a fundamental domain of $\Ga$.}
	Let $A_{i}, B_{i}, i=1, \dots g,$ and $S_{i}, i=1, \dots m,$ be the hyperbolic and parabolic generators of $\G$, respectively (no elliptic ones by assumption).
	Let $\al_i, \al'_i, \beta_i, \beta'_i, i=1, \dots g,$ and $\delta_j, \delta_j', j=1, \dots m$ be the sides of $ F_\Si$.
	
{	As
	in the beginning of \secref{sec:3},
	we have the following relations:
	\[A_{i}^*\al_{i} = -\al_{i}', B_{i}^*\beta_{i} = -\beta_{i}', i=1, \ldots g,\text{ and } S_{j}^*\delta_j = -\delta_j', j=1, \dots m.\] 
	as well as the decomposition
	\begin{align} \label{dD'} \p F_\Si=\sum_{i=1}^g(\al_i + \al'_i + \beta_i + \beta'_i)+\sum_{j=1}^m (\delta_j + \delta_j').
	\end{align}Here the plus sign denotes disjoint union.}

{	Let $A$ be an equivariant $1$-form on $F_\Si$ that
	corresponds to the connection $a$ on $E\to\Si$,
	as in Sections \ref{sec:2.2}--\ref{sec:2.4}.
	By definition \eqref{deg}, the theorem is proved once we establish
	$$\frac{1}{2\pi}\int_{\di F_\Si}A=c_1(\si).$$
}
	
	Since the  $1$-form  $A$ is gauge-equivariant, we have 
	\begin{equation}\label{gauge-per-Gam} 
		\g^*A = A - i \rho_\g^{-1} d \rho_\g ,\ \forall \g\in \Ga,\,z\in \di F_\Si. 
	\end{equation}
	This, together with the definition $\rho (\g, z)=e^{i \si_\g (z)}$, implies
	\[\int_{\al_{i} +\al_{i}'} A= \int_{\al_{i} } (A- A_{i}^*A)=- \int_{\al_{i} } d \si_{A_{i}}=\si_{A_{i}}(v_i)-\si_{A_{i}}(v_{i+1}) \quad (i=1, \dots g),\] 
	where $v_{i}, v_{i+1}$ are the vertices of $F_\Si$ spanned by the side $\al_{i}$. Similarly we proceed with $\int_{\beta_{i} +\beta_{i}'} A$ and $\int_{\delta_j +\delta_j'} A$. In the last case, since the gauge exponent $\si_{S_{i}}(z)$ is independent of $z$, we have $\int_{\delta_j +\delta_j'} A=0$, so these term do not contribute.
	
	Summing up these contributions and using \eqref{dD'}, we arrive at 
	\[\int_{\p F_\Si} A=\sum_1^g [\si_{A_{i}}(v_i)-\si_{A_{i}}(v_{i+1}) +\si_{B_{i}}(w_i)-\si_{B_{i}}(w_{i+1})],\]
	where $v_{i}, v_{i+1}$ and  $w_{i}, w_{i+1}$  
	are the vertices of $F_\si$ spanned by $\al_{i}$ and $\beta_{i}$,  
	respectively. 
	Note that the vertices $v_{i}, v_{i+1}$, $w_{i}, w_{i+1}$ and $u_{i}, u_{i+1}$ are related as  $w_{i}=v_{i+1},  w_{i+1}=A_i v_{i+1} $ and $v_{i+1}= B_{i}^{-1}v_{i}$, which gives 
	
\begin{equation}\label{C.1}
		\frac{1}{2\pi}\int_{\p F_\Si} A =\frac{1}{2\pi}\sum_1^g [\si_{A_{i}}(v_i)-\si_{A_{i}}(B_{i}^{-1}v_{i}) +\si_{B_{i}}(B_{i}^{-1}v_{i})-\si_{B_{i}}(A_i B_{i}^{-1}v_{i})].
\end{equation}
	By construction, $$v_{i}=B_i A_{i}^{-1}B_{i}^{-1} C_{i+1}\dots C_g z_0= A_{i}^{-1} C_{i}\dots C_g z_0,$$
	with $C_{i}:= A_{i} B_i A_{i}^{-1}B_{i}^{-1}\text{  and some } z_0\in F_\Si.$
	Hence, r.h.s. of \eqref{C.1} agrees with the definition of the first Chern number
	$c_1$ in \eqref{c1g}. This proves the theorem. \end{proof}  

{\section{ Proof of Theorem \ref{Kras}}\label{secPfKras}
	
	One reduces to a scalar bifurcation problem by considering the scalar-valued function
	\begin{eqnarray} \label{phicond}
		\phi(\ell, s) &=& \frac1{|s|^2}\langle s, \ell s + H(\ell, s)\rangle.
	\end{eqnarray}
	Manifestly, $\phi(0,0) = 0$ and $\partial \phi (0,0)/\partial \ell = 1$. It then follows from the implicit function theorem that there exists a unique function $\ell(s)$,
	for small $s$, such that $\phi(\ell(s), s) = 0$. {Set $\hat{E}(s) = E(\ell(s), s)$, where, recall, $E(\ell,s)=\ell s+H(\ell,s)$,  and let  $S_\epsilon:=\Set{s\in\Rb^n:\abs{s}=\eps}$ be the $\epsilon$-sphere in $\mathbb{R}^n$.}
	Then $\hat{E}(s)$, regarded as a vector field on $\mathbb{R}^n$, is everywhere tangent to $S_\epsilon$. To see this, simply note that 
	\begin{eqnarray*}
		0 &=& \phi(\ell(s), s)\\
		&=& \frac1{|s|^2}\langle s, \ell(s) s + H(\ell(s), s)\rangle\\
		&=& \frac1{|s|^2}\langle s, \hat{E}(s)\rangle,
	\end{eqnarray*}
	which precisely states that the vector field $\hat E(s)$ is everywhere tangent to the sphere. Since $n$ is odd, the sphere is even dimensional. But every 
	vector field on an even dimensional sphere has as zero vector, i.e. $\hat E(s)=0$ has a solution on $S_\eps$. The theorem follows.}



\end{document}